\documentclass{amsart}
\usepackage[all]{xy}
\usepackage{color}
\usepackage{amsthm}
\usepackage{amssymb} 
\usepackage[colorlinks=true]{hyperref}

\numberwithin{equation}{section}

\newtheorem{theorem}[equation]{Theorem}
\newtheorem*{theorem*}{Theorem}
\newtheorem{lemma}[equation]{Lemma}

\newtheorem*{conjecture*}{Mamma Conjecture}
\newtheorem*{conjecture1*}{Mamma Conjecture (revisited)}
\newtheorem{proposition}[equation]{Proposition}
\newtheorem{corollary}[equation]{Corollary}
\newtheorem*{corollary*}{Corollary}

\theoremstyle{remark}
\newtheorem{definition}[equation]{Definition}

\newtheorem{notation}[equation]{Notation}

\theoremstyle{remark}
\newtheorem{remark}[equation]{Remark}

\newcommand{\cA}{{\mathcal A}}
\newcommand{\cB}{{\mathcal B}}
\newcommand{\cC}{{\mathcal C}}
\newcommand{\cD}{{\mathcal D}}

\newcommand{\cN}{{\mathcal N}}
\newcommand{\cO}{{\mathcal O}}

\newcommand{\cU}{{\mathcal U}}
\newcommand{\cV}{{\mathcal V}}

\newcommand{\cZ}{{\mathcal Z}}

\newcommand{\onil}{\otimes_\mathrm{nil}}
\newcommand{\sVect}{\mathrm{sVect}}

\newcommand{\bbF}{\mathbb{F}}
\newcommand{\bbG}{\mathbb{G}}

\newcommand{\bbL}{\mathbb{L}}
\newcommand{\bbN}{\mathbb{N}}
\newcommand{\bbP}{\mathbb{P}}

\newcommand{\bbQ}{\mathbb{Q}}
\newcommand{\bbZ}{\mathbb{Z}}

\DeclareMathOperator{\SmProj}{SmProj} 
\DeclareMathOperator{\DM}{DM} 

\DeclareMathOperator{\Id}{Id}
\DeclareMathOperator{\id}{id}

\DeclareMathOperator{\NChow}{NChow} 
\DeclareMathOperator{\NVoev}{NVoev} 
\DeclareMathOperator{\NNum}{NNum} 
\DeclareMathOperator{\NHom}{NHom} 
\DeclareMathOperator{\Chow}{Chow} 
\DeclareMathOperator{\AM}{AM} 
\DeclareMathOperator{\MAM}{MAM} 
\DeclareMathOperator{\NAM}{NAM} 
\DeclareMathOperator{\NMAM}{NMAM} 
\DeclareMathOperator{\Num}{Num} 

\DeclareMathOperator{\K}{K}

\DeclareMathOperator{\Mix}{KMM} 
\DeclareMathOperator{\KPM}{KPM} 
\DeclareMathOperator{\KTM}{KTM} 
\DeclareMathOperator{\Fun}{Fun} 

\newcommand{\dgcat}{\mathsf{dgcat}}
\newcommand{\sdgcat}{\mathsf{sdgcat}}

\newcommand{\perf}{\mathsf{perf}}

\newcommand{\dg}{\mathsf{dg}}

\newcommand{\Hom}{\mathrm{Hom}}
\newcommand{\End}{\mathrm{End}}

\newcommand{\rep}{\mathsf{rep}}

\newcommand{\dgHo}{\mathsf{H}^0}

\newcommand{\Hmo}{\mathsf{Hmo}}
\newcommand{\sHmo}{\mathsf{sHmo}}
\newcommand{\op}{\mathsf{op}}

\newcommand{\too}{\longrightarrow}

\newcommand{\ie}{\textsl{i.e.}\ }

\begin{document}

\title[Noncommutative Artin motives]{Noncommutative Artin motives}
\author{Matilde Marcolli and Gon{\c c}alo~Tabuada}

\address{Matilde Marcolli, Mathematics Department, Mail Code 253-37, Caltech, 1200 E.~California Blvd. Pasadena, CA 91125, USA}
\email{matilde@caltech.edu} 
\urladdr{http://www.its.caltech.edu/~matilde}

\address{Gon{\c c}alo Tabuada, Department of Mathematics, MIT, Cambridge, MA 02139, USA}
\email{tabuada@math.mit.edu}
\urladdr{http://math.mit.edu/~tabuada}

\date{\today}

\thanks{The first named author was supported by NSF grants
DMS-0901221, DMS-1007207, DMS-1201512, and PHY-1205440.
The second named author was supported by the NEC Award-$2742738$.}

\abstract{In this article we introduce the category of noncommutative Artin motives as well as the category of noncommutative mixed Artin motives. In the pure world, we start by proving that the classical category $\AM(k)_\bbQ$ of Artin motives (over a base field $k$) can be characterized as the largest category inside Chow motives which fully-embeds into noncommutative Chow motives. Making use of a refined bridge between pure motives and noncommutative pure motives we then show that the image of this full embedding, which we call the category $\NAM(k)_\bbQ$ of {\em noncommutative Artin motives}, is invariant under the different equivalence relations and modification of the symmetry isomorphism constraints. As an application, we recover the absolute Galois group $\mathrm{Gal}(\overline{k}/k)$ from the Tannakian formalism applied to $\NAM(k)_\bbQ$. Then, we develop the base-change formalism in the world of noncommutative pure motives. As an application, we obtain new tools for the study of motivic decompositions and Schur/Kimura finiteness. Making use of this theory of base-change we then construct a short exact sequence relating $\mathrm{Gal}(\overline{k}/k)$ with the noncommutative motivic Galois groups of $k$ and $\overline{k}$. Finally, we describe a precise relationship between this short exact sequence and the one constructed by Deligne-Milne. In the mixed world, we introduce the triangulated category $\NMAM(k)_\bbQ$ of {\em noncommutative mixed Artin motives} and construct a faithful functor from the classical category $\MAM(k)_\bbQ$ of mixed Artin motives to it. When $k$ is a finite field this functor is an equivalence. On the other hand, when $k$ is of characteristic zero $\NMAM(k)_\bbQ$ is much richer than $\MAM(k)_\bbQ$ since its higher Ext-groups encode all the (rationalized) higher algebraic $K$-theory of finite {\'e}tale $k$-schemes. In the appendix we establish a general result about short exact sequences of Galois groups which is of independent interest. As an application, we obtain a new proof of Deligne-Milne's short exact sequence.
}}
\maketitle
\vskip-\baselineskip
\vskip-\baselineskip

\section{Introduction}
In this article we further the theory of noncommutative motives, initiated in \cite{CT,CT1,Semi,Kontsevich,Galois,uncond,Duke,IMRN,CvsNC, Weight}, by introducing the categories of noncommutative (mixed) Artin motives and exploring many of its features.

\subsection*{Artin motives}
Recall from \cite[\S4]{Andre} the construction of the category $\Chow(k)_\bbQ$ of Chow motives (over a base field $k$ and with rational coefficients) and of the functor
\begin{equation}\label{eq:functor-Chow}
M:\mathrm{SmProj}(k)^\op \too \Chow(k)_\bbQ 
\end{equation}
defined on smooth projective $k$-schemes.
Given any full subcategory $\cV$ of $\mathrm{SmProj}(k)$ (stable under finite products and disjoint unions), we can then consider the smallest additive rigid idempotent complete subcategory $\Chow(k)_\bbQ^\cV$ of $\Chow(k)_\bbQ$ generated by the objects $M(Z), Z \in \cV$. For example, by taking for $\cV$ the finite {\'e}tale $k$-schemes we obtain the category $\AM(k)_\bbQ$ of {\em Artin motives}; see \cite[\S4.1.6.1]{Andre}. This abelian semi-simple category plays a central role in the beautiful relationship between motives and Galois theory. Concretely, $\AM(k)_\bbQ$ is invariant under the different equivalence relations (rational, homological, numerical, etc) on algebraic cycles and under the modification of the symmetry isomorphism constraints. Moreover, by considering it inside the Tannakian category $\Num^\dagger(k)_\bbQ$ of numerical motives, the associated motivic Galois group $\mathrm{Gal}(\AM(k)_\bbQ)$ agrees with the absolute Galois group $\mathrm{Gal}(\overline{k}/k)$ of $k$; see \cite[\S6.2.5]{Andre}. For this reason, $\mathrm{Gal}(\Num^\dagger(k)_\bbQ)$ is commonly interpreted as an ``higher dimensional Galois group''. As proved by Deligne-Milne in \cite[\S6]{DelMil}, these groups fit into a short exact sequence
\begin{equation}\label{eq:shortexact}
1 \to \mathrm{Gal}(\Num^\dagger(\overline{k})_\bbQ) \to \mathrm{Gal}(\Num^\dagger(k)_\bbQ) \to \mathrm{Gal}(\overline{k}/k) \to 1
\end{equation} 
which expresses the absolute Galois group as the quotient of the base-change along the algebraic closure; consult also Serre \cite[\S6]{Serre}.
\subsection*{Motivating questions}
As explained in the survey article \cite{survey}, the above functor \eqref{eq:functor-Chow} admits a noncommutative analogue
\begin{equation*}
\cU:\sdgcat(k) \too \NChow(k)_\bbQ
\end{equation*}
defined on saturated dg categories in the sense of Kontsevich; see \S\ref{sub:saturated} and \S\ref{sub:Chow}. Moreover, as proved in \cite[Thm.~1.1]{CvsNC}, there exists a $\bbQ$-linear $\otimes$-functor $\Phi$ making the following diagram commute
\begin{equation}\label{eq:bridge}
\xymatrix{
\mathrm{SmProj}(k)^\op \ar[d]_-{M} \ar[rr]^-{\cD_\perf^\dg(-)} && \sdgcat(k) \ar[d]^{\cU} \\
\Chow(k)_\bbQ \ar[rr]_-{\Phi} && \NChow(k)_\bbQ\,,
}
\end{equation}
where $\cD_\perf^\dg(Z)$ denotes the (unique) dg enhancement of the derived category of perfect complexes of $\cO_Z$-modules (see~\cite{LO}\cite[Example~5.5(i)]{CT1}). Hence, given any category $\cV$ as above, one can consider in a similar way the smallest additive rigid idempotent complete subcategory $\NChow(k)_\bbQ^\cV$ of $\NChow(k)_\bbQ$ generated by the objects $\cU(\cD_\perf^\dg(Z)), Z \in \cV$. By the above commutative diagram one obtains then a well-defined $\bbQ$-linear $\otimes$-functor $\Chow(k)_\bbQ^\cV \to \NChow(k)^\cV_\bbQ$. The above mentioned results on Artin motives lead us naturally to the following questions:

\smallbreak

Question I: \textit{When is the functor $\Chow(k)_\bbQ^\cV \to \NChow(k)^\cV_\bbQ$ an equivalence ?} 

Question II: \textit{Is the noncommutative analogue $\NAM(k)_\bbQ$ of the category of Artin motives also invariant under the different equivalence relations and modification of the symmetry isomorphism constraints ?}

Question III: \textit{Does the noncommutative motivic Galois group $\mathrm{Gal}(\NAM(k)_\bbQ)$ also agree with the absolute Galois group of $k$ ?}

Question IV: \textit{Does the above sequence \eqref{eq:shortexact} admit a noncommutative analogue~?}

\subsection*{Statement of results}
All the above considerations hold more generally with $\bbQ$ replaced by any field extension $F/\bbQ$. In particular, we have a well-defined functor
\begin{equation}\label{eq:functor-central}
\Chow(k)^\cV_F \too \NChow(k)_F^\cV\,.
\end{equation}
The answer to Question I is the following:
\begin{theorem}\label{thm:main1}
The functor \eqref{eq:functor-central} is a (tensor) equivalence if and only if $\cV$ is contained in the category of finite {\'e}tale $k$-schemes.
\end{theorem}
Informally speaking, Theorem~\ref{thm:main1} characterizes the category of Artin motives as being precisely the largest subcategory of Chow motives which fully-embeds into noncommutative Chow motives. Let us denote by $\NAM(k)_F$ the image of this full embedding and call it the category of {\em noncommutative Artin motives}. The noncommutative analogues of the classical equivalence relations on algebraic cycles were introduced by the authors in \cite{Semi,Kontsevich,Galois,uncond}. Concretely, we have a sequence of $F$-linear full additive $\otimes$-functors
\begin{equation}\label{eq:composed}
\NChow(k)_F \too \NVoev(k)_F \too \NHom(k)_F \too \NNum(k)_F\,,
\end{equation}
where $\NVoev(k)_F$ denotes the category of noncommutative Voevodsky motives, $\NHom(k)_F$ the category of noncommutative homological motives, and $\NNum(k)_F$ the category of noncommutative numerical motives; consult \S\ref{sec:NCmotives} for details. Making use of a refined bridge between the commutative and the noncommutative world, we answer affirmatively to the first part of Question II:
\begin{theorem}\label{thm:main2}
The composed functor \eqref{eq:composed} becomes fully-faithful when restricted to the category $\NAM(k)_F$ of noncommutative Artin motives.
\end{theorem}
As proved in \cite{Semi,Kontsevich,Galois}, the category $\NNum(k)_F$ is abelian semi-simple when $k$ and $F$ have the same characteristic. Assuming the noncommutative standard conjecture $C_{NC}$ (= K{\"u}nneth) and that $k$ is of characteristic zero, this category can be made into a Tannakian category $\NNum^\dagger(k)_F$ by modifying its symmetry isomorphism constraints; see \cite[Thm.~1.4]{Galois}. Moreover, assuming the noncommutative standard conjecture $D_{NC}$ (= homological equals numerical) and that $F$ is a field extension of $k$, $\NNum^\dagger(k)_F$ becomes a neutral Tannakian category with fiber functor given by periodic cyclic homology; see \cite[Thm.~1.6]{Galois}. As a consequence, one obtains well-defined {\em noncommutative motivic Galois groups} $\mathrm{Gal}(\NNum^\dagger(k)_F)$, which are moreover pro-reductive (\ie their unipotent radicals are trivial) since $\NNum^\dagger(k)_F$ is semi-simple.
\begin{theorem}\label{thm:main3}
\begin{itemize}
\item[(i)] We have an equality $\NAM^\dagger(k)_F = \NAM(k)_F$ of categories;
\item[(ii)] The inclusion $\NAM(k)_F \subset \NNum^\dagger(k)_F$ gives rise to a surjection
\begin{equation*}
\mathrm{Gal}(\NNum^\dagger(k)_F) \twoheadrightarrow \mathrm{Gal}(\NAM(k)_F)\,.
\end{equation*}
\item[(iii)] The noncommutative motivic Galois group $\mathrm{Gal}(\NAM(k)_F)$ is naturally isomorphism to the absolute Galois group of $k$.
\end{itemize}
\end{theorem}
Item (i) answers the second part of Question II, while item (iii) provides an affirmative answer to Question III. Hence, as in the commutative world, $\mathrm{Gal}(\NNum^\dagger(k)_F)$ can also be understood as an ``higher dimensional Galois group''. Its precise relationship with the classical motivic Galois group and with the multiplicative group $\bbG_m$ is described in \cite[Thm.~1.7]{Galois}. Our answer to Question IV is the following:
\begin{theorem}\label{thm:main4}
We have a short exact sequence
\begin{equation}\label{eq:short-new}
 1 \to \mathrm{Gal}(\NNum^\dagger(\overline{k})_F) \stackrel{I}{\too} \mathrm{Gal}(\NNum^\dagger(k)_F) \stackrel{P}{\too} \mathrm{Gal}(\overline{k}/k) \to 1\,.
\end{equation}
The map $P$ is induced by the inclusion $\NAM(k)_F \subset \NNum^\dagger(k)_F$ and the map $I$ by the base-change functor $-\otimes_k \overline{k}: \NNum^\dagger(k)_F \to \NNum^\dagger(\overline{k})_F$; see \S\ref{sec:basechange}.
\end{theorem}
Deligne-Milne's proof of the short exact sequence \eqref{eq:shortexact} makes full use of ``commutative arguments'' which don't seem to admit noncommutative analogues. The proof of Theorem~\ref{thm:main4} is hence not only different but moreover much more conceptual from a categorical viewpoint; see \S\ref{sec:proof4}. By extracting the key ingredients of this latter proof we have established in Appendix~\ref{appendix:Galois} a general result about short exact sequence of Galois groups which is of independent interest; see Theorem~\ref{thm:general}. As an application, we obtain a new proof of \eqref{eq:shortexact} which circumvents the ``commutative arguments'' of Deligne-Milne; see \S\ref{sub:new-proof}.

The theory of base-change in the noncommutative world (needed for the proof of Theorem~\ref{thm:main4}) was developed in \S\ref{sub:basechange1}-\ref{sub:basechange2}. Its compatibility with the classical base-change mechanism is proved in \S\ref{sub:compatibility}. As an application, we obtain new tools for the study of motivic decompositions and Schur/Kimura finiteness. Concretely, Corollaries~\ref{cor:application-1} and \ref{cor:application-2} can be used in order to prove that certain Chow motives are of Lefschetz type or Schur/Kimura-finite. 

Finally, our last result in the pure world describes a precise relationship between Deligne-Milne's short exact sequence \eqref{eq:shortexact} and its noncommutative analogue \eqref{eq:short-new}.
\begin{theorem}\label{thm:diagram}
We have a well-defined commutative diagram
\begin{equation}\label{eq:diagram}
\xymatrix@C=1.5em@R=2em{
& 1 & 1 & & \\
& \bbG_m \ar[u] \ar@{=}[r] & \bbG_m \ar[u] & & \\
1 \ar[r] & \mathrm{Gal}(\Num^\dagger(\overline{k})_F) \ar[r]^-I \ar[u]  \ar@{}[dr]|{\llcorner} & \mathrm{Gal}(\Num^\dagger(k)_F) \ar[r]^-P \ar[u] & \mathrm{Gal}(\overline{k}/k) \ar[r] & 1 \\
1 \ar[r] & \mathrm{Gal}(\NNum^\dagger(\overline{k})_F)  \ar[u] \ar[r]_-I & \mathrm{Gal}(\NNum^\dagger(k)_F) \ar[u] \ar[r]_-P & \mathrm{Gal}(\overline{k}/k) \ar[r] \ar@{=}[u] &1\,, 
}
\end{equation}
where each row and each column is exact. In particular, the lower left-hand-side square is cartesian.
\end{theorem}
Intuitively speaking, Theorem~\ref{thm:diagram} show us that the short exact sequence \eqref{eq:short-new} maps towards the short exact sequence \eqref{eq:shortexact} and that the default of surjectivity of this map is precisely the multiplicative group $\bbG_m$.
\subsection*{Noncommutative mixed Artin motives}
Recall from \cite[\S16]{Andre} the construction of the triangulated category $\DM_{gm}(k)_\bbQ$ of Voevodsky's mixed motives (over a perfect field $k$) and of the functor $M_{gm}: \mathrm{Sm}(k) \to \DM_{gm}(k)_\bbQ$ defined on smooth $k$-schemes. As explained in \cite[\S5]{survey}, this functor admits a noncommutative analogue $\cU_M: \sdgcat(k) \to \Mix(k)_\bbQ$ with values in the triangulated category of Kontsevich's mixed motives; see \S\ref{sub:mixed}. Hence, by performing the same constructions as above (using triangulated categories instead of additive ones), we obtain the classical category $\MAM(k)_\bbQ$ of mixed Artin motives (see \cite[\S1]{Wildehaus}) and also its noncommutative analogue $\NMAM(k)_\bbQ$, which we call the category of {\em noncommutative mixed Artin motives}. These categories are related as follows:

\begin{theorem}\label{thm:main5}
\begin{itemize}
\item[(i)]
There exists a faithful $\bbQ$-linear triangulated $\otimes$-functor $\Psi$ making the following diagram commute
\begin{equation}\label{eq:diag-mixed}
\xymatrix{
\{\textrm{finite {\'e}tale $k$-schemes}\}^\op \ar[d]_-{M_{gm}^\vee} \ar[r]^-{\cD_\perf^\dg(-)} & \sdgcat(k) \ar[d]^{\cU_M} \\
\MAM(k)_\bbQ \ar[r]_-{\Psi} & \NMAM(k)_\bbQ \subset \Mix(k)_\bbQ\,,  \qquad \qquad
}
\end{equation}
where $(-)^\vee$ stands for the duality functor. 
\item[(ii)]When $k=\bbF_q$ the functor $\Psi$ is moreover an equivalence. As a consequence $\NMAM(\bbF_q)_\bbQ \simeq \cD^b(\mathrm{Rep}_\bbQ(\widehat{\bbZ}))$, where $\mathrm{Rep}_\bbQ(\widehat{\bbZ})$ denotes the category of finite dimensional $\bbQ$-linear representations of $\widehat{\bbZ}:=\mathrm{lim}_n \bbZ/n\bbZ$.
\item[(iii)] When $\bbQ\subseteq k$ the functor $\Psi$ is {\em not} full. For instance, the non-trivial maps
\begin{eqnarray*}
\Hom_{\NMAM(\bbQ)_\bbQ}(\Psi M_{gm}^\vee(\mathrm{spec}(\bbQ)),\Psi M_{gm}^\vee(\mathrm{spec}(\bbQ))[-n])\simeq \bbQ & n \equiv 5 \,\,(\textrm{mod} \,\,8)&
\end{eqnarray*}
are not in the image of $\Psi$.
\end{itemize}
\end{theorem}
In contrast with the world of pure motives, Theorem~\ref{thm:main5} shows us that in characteristic zero the category $\NMAM(k)_\bbQ$ is much richer than $\MAM(k)_\bbQ$. Informally speaking, $\NMAM(k)_\bbQ$ keeps track (in terms of Ext-groups) of {\em all} the higher algebraic $K$-theory of finite {\'e}tale $k$-schemes.
\subsection*{Future research}
Let us denote $\AM(k)^{\mathrm{pro}}_\bbQ$ the category of pro-objects in $\AM(k)_\bbQ$. Recall from \cite{CM,CCM,Mar,Yalk} that an {\em algebraic endomotive} consists of a triple $(M,S,\mu)$, where $M$ is an object of $\AM(k)_\bbQ^{\mathrm{pro}}$, $S$ is a countable abelian semigroup contained in $\End_{\AM(k)_\bbQ^{\mathrm{pro}}}(M)$, and $\mu:S \to \bbN$ is a semigroup homomorphism. Recall also from {\em loc. cit.} that every algebraic endomotive $(M,S,\mu)$ gives rise to a ``Riemann'' zeta function
$$ Z((M,S,\mu);\beta):= \sum_{s \in S} \mu(s)^{-\beta}\,.$$
Thanks to the above Theorem~\ref{thm:main2} one obtains the same theory of algebraic endomotives if one replaces $\AM(k)_\bbQ^{\mathrm{pro}}$ by $\NAM(k)_\bbQ^{\mathrm{pro}}$. On the other hand, Theorem~\ref{thm:main5}(iii) shows us that if one replaces $\AM(k)_\bbQ^{\mathrm{pro}}$ by $\NMAM(k)_\bbQ^{\mathrm{pro}}$ (with $k$ of characteristic zero), then one obtains a richer theory of algebraic endomotives. In particular, all the higher algebraic $K$-theory of finite {\'e}tale $k$-schemes is available in this new theory. This lead us naturally to the following questions which are the subject of future research.

\vspace{0.1cm}

{\it Question: Can the higher algebraic $K$-theory of finite {\'e}tale $k$-schemes be used to construct an interesting enrichment of the Bost-Connes endomotive ?}

{\it Question: Which geometric and arithmetic information can be captured by the new zeta functions ?}

\medbreak

\noindent\textbf{Acknowledgments:} The authors are very grateful to Michael Artin and Yuri Manin for motivating questions, to Joseph Ayoub, Dmitry Kaledin and Burt Totaro for fruitful discussions, to Bernhard Keller for precise comments on a previous draft, and to Yves Andr{\'e} and Bruno Kahn for useful e-mail exchanges. They are also grateful to the anonymous referee for his/her comments.

\medbreak

\noindent\textit{Conventions:}  Throughout the article we will reserve the letter $k$ for the base field and the letter $F$ for the field of coefficients. The pseudo-abelian envelope construction will be denoted by $(-)^\natural$ and all adjunctions will be displayed vertically with the left (resp. right) adjoint on the left (right) hand-side.
\section{Differential graded categories}\label{sec:dg}
For a survey article on dg categories consult Keller's ICM address \cite{ICM} . Let $\cC(k)$ the category of (unbounded) cochain complexes of $k$-vector spaces. A {\em differential graded (=dg) category $\cA$} is a category enriched over $\cC(k)$. Concretely, the morphisms sets $\cA(x,y)$ are complexes of $k$-vector spaces and the composition law fulfills the Leibniz rule: $d(f\circ g)=d(f)\circ g+(-1)^{\textrm{deg}(f)}f\circ d(g)$. The category of dg categories will be denoted by $\dgcat(k)$.

Let $\cA$ be a (fixed) dg category. Its {\em opposite} dg category $\cA^\op$ has the same objects and complexes of morphisms given by $\cA^\op(x,y):=\cA(y,x)$. The $k$-linear category $\dgHo(\cA)$ has the same objects as $\cA$ and morphisms given by $\dgHo(\cA)(x,y):= \textrm{H}^0\cA(x,y)$, where $\mathrm{H}^0$ denotes $0^{\mathrm{th}}$ cohomology.
\begin{definition}\label{def:quasi-eq}
A dg functor $G:\cA \to \cB$ is a called a {\em quasi-equivalence} if:
\begin{itemize}
\item[(i)] we have a quasi-isomorphism $\cA(x,y) \stackrel{\sim}{\to} \cB(Gx,Gy)$ for all objects $x$ and $y$;
\item[(ii)] the associated functor $\dgHo(G): \dgHo(\cA) \to \dgHo(\cB)$ is essentially surjective.
\end{itemize}
\end{definition}
A {\em right dg $\cA$-module} $X$ (or simply a $\cA$-module) is a dg functor $X: \cA^\op \to \cC_\dg(k)$ with values in the dg category of complexes of $k$-vector spaces. We will denote by $\cC(\cA)$ the category of $\cA$-modules. Recall from \cite[\S3]{ICM} that $\cC(\cA)$ carries a {\em projective} model structure. Moreover, the differential graded structure of $\cC_\dg(k)$ makes $\cC(\cA)$ naturally into a dg category $\cC_\dg(\cA)$. Let $\cD(\cA)$ be the {\em derived category} of $\cA$, \ie the localization of $\cC(\cA)$ with respect to the class of quasi-isomorphisms. Its full triangulated subcategory of compact objects (see \cite[Def.~4.2.7]{Neeman}) will be denoted by $\cD_c(\cA)$.
\begin{definition}
A dg functor $G:\cA \to \cB$ is called a {\em derived Morita equivalence} if the (derived) extension of scalars functor $\bbL G_!:\cD_c(\cA) \stackrel{\sim}{\to} \cD_c(\cB)$ is an equivalence of categories; see \cite[\S4.6]{ICM}.
\end{definition}
\begin{notation}\label{not:Yoneda}
We will denote by $\widehat{\cA}$ the full dg subcategory of $\cC_\dg(\cA)$ consisting of those cofibrant $\cA$-modules which become compact in $\cD(\cA)$. Note that $\dgHo(\widehat{\cA})\simeq \cD_c(\cA)$. The assignment $\cA \mapsto \widehat{\cA}$ is functorial and we have a Yoneda dg functor $h:\cA \to \widehat{\cA}, x \mapsto \cA(-,x)$.
\end{notation}

\begin{proposition}\label{prop:char}
The following assertions are equivalent:
\begin{itemize}
\item[(i)] A dg functor $G:\cA \to \cB$ is a derived Morita equivalence;
\item[(ii)] The associated dg functor $\widehat{G}: \widehat{\cA} \to \widehat{\cB}$ is a quasi-equivalence;
\item[(iii)] The dg functor $G$ satisfies condition (i) of Definition~\ref{def:quasi-eq} and every object of $\dgHo(\cB) \subset \cD_c(\cB)$ is in the essential image of the functor $\bbL G_!: \cD_c(\cA) \to \cD_c(\cB)$.
\end{itemize}
\end{proposition}
\begin{proof} The implication (i) $\Rightarrow$ (iii) follows from the natural isomorphisms $\textrm{H}^n\cA(x,y)\simeq \Hom_{\cD_c(\cA)}(h(x),h(y)[n])$, $n \in \bbZ$. If $G$ satisfies condition (i) of Definition~\ref{def:quasi-eq}, then by construction so it does the dg functor $\widehat{G}$. The conditions of item (iii) of Proposition~\ref{prop:char} imply that the functor $\bbL G_!$ is essentially surjective. Since $\dgHo(\widehat{G})\simeq \bbL G_!$, we conclude then that $\widehat{G}$ is a quasi-equivalence. This shows the implication (iii) $ \Rightarrow$ (ii). Finally, note that if $G$ is a quasi-equivalence then $\dgHo(G)$ is an equivalence. As a consequence, the implication (ii) $\Rightarrow$ (i) follows from the natural isomorphism $\bbL G_!\simeq \dgHo(\widehat{G})$.
\end{proof}
As proved in \cite[Thm.~5.3]{IMRN}, the category $\dgcat(k)$ carries a Quillen model structure whose weak equivalences are the {\em derived Morita equivalences}. The homotopy category hence obtained will be denoted by $\Hmo(k)$. The tensor product of $k$-algebras extends naturally to dg categories, giving rise to a symmetric monoidal structure $-\otimes_k-$ on $\dgcat(k)$, which descends to the homotopy category. Finally, recall from \cite[\S3.8]{ICM} that a {\em $\cA\text{-}\cB$-bimodule} $X$ is a dg functor $X:\cA\otimes_k \cB^\op \to \cC_\dg(k)$, or in other words a $(\cA^\op \otimes_k \cB)$-module.
\subsection{Saturated dg categories}\label{sub:saturated}
Following Kontsevich~\cite{IAS,Miami}, a dg category $\cA$ is called {\em saturated} if the $\cA\text{-}\cA$-bimodule
\begin{eqnarray*}
\cA(-,-): \cA \otimes_k \cA^\op \to \cC_\dg(k) && (x,y) \mapsto \cA(x,y)
\end{eqnarray*}
belongs to $\cD_c(\cA^\op \otimes_k \cA)$ and $\sum_i \mathrm{dim}\,\textrm{H}^i\cA(x,y)<\infty$ for all objects $x$ and $y$. As proved in \cite[Thm.~5.8]{CT1}, the saturated dg categories can be conceptually characterized as being precisely the dualizable (or rigid) objects of the symmetric monoidal category $\Hmo(k)$. Moreover, they are always derived Morita equivalent to saturated dg algebras.
\section{Noncommutative motives}\label{sec:NCmotives}
In this section we recall the construction of the categories of noncommutative pure and mixed motives. For further details consult the survey article \cite{survey}.
\subsection{Noncommutative Chow motives}\label{sub:Chow}
The rigid category $\NChow(k)_F$ of {\em noncommutative Chow motives} is the pseudo-abelian envelope of the category:
\begin{itemize}
\item[(i)] whose objects are the saturated dg categories;
\item[(ii)]  whose morphisms from $\cA$ to $\cB$ are given by the $F$-linearized Grothendieck group $K_0(\cA^\op \otimes_k \cB)_F$;
\item[(iii)] whose composition law is induced by the (derived) tensor product of bimodules.
\end{itemize}
The functor $\cU:\sdgcat(k) \to \NChow(k)_F$ mentioned in the introduction sends a dg functor to the class (in the $F$-linearized Grothendieck group) of the naturally associated bimodule. The symmetric monoidal structure on $\NChow(k)_F$ is induced by the tensor product of dg categories.
\subsection{Noncommutative Voevodsky motives}\label{sub:Voevodsky}
The rigid category $\NVoev(k)_F$ of {\em noncommutative Voevodsky motives} was introduced by the authors in \cite[\S3]{uncond}. It is defined as the pseudo-abelian envelope of the quotient category $\NChow(k)_F\!/\onil$, where $\onil$ denotes the classical $\otimes$-nilpotence ideal.
\subsection{Noncommutative homological motives}\label{sub:homological}
As proved in \cite[Thm.~7.2]{Galois}, periodic cyclic homology $HP$ gives rise to a well-defined $F$-linear $\otimes$-functor
\begin{equation}\label{eq:functor-HP}
\overline{HP_\ast}: \NChow(k)_F \too \sVect(K)
\end{equation}
with values in the category of super $K$-vector spaces, where $K=F$ when $F$ is a field extension of $k$ and $K=k$ when $k$ is a field extension of $F$. The rigid category $\NHom(k)_F$ of {\em noncommutative motives} was introduced by the authors in \cite[\S10]{Galois} as the pseudo-abelian envelope of the quotient category $\NChow(k)_F/\mathrm{Ker}(\overline{HP_\ast})$.
\subsection{Noncommutative numerical motives}\label{sub:numerical}
The rigid category $\NNum(k)_F$ of {\em noncommutative numerical motives} is defined as the pseudo-abelian envelope of the quotient category $\NChow(k)_F/\cN$, where $\cN$ is the largest $\otimes$-ideal of $\NChow(k)_F$ (distinct from the entire category). As proved in \cite{Kontsevich}, $\NNum(k)_F$ is $\otimes$-equivalent to the pseudo-abelian envelope of the category:
\begin{itemize}
\item[(i)] whose objects are the saturated dg algebras;
\item[(ii)] whose morphisms from $A$ and $B$ are given by $K_0(A^\op \otimes_k B)_F\!/\mathrm{Ker}(\chi)$, where $\chi$ is the $F$-linearized bilinear form on $K_0(A^\op \otimes_kB)_F$ given by
\begin{eqnarray}\label{eq:bilinear}
([X],[Y]) & \mapsto & \sum_{n \in \bbZ} (-1)^n \mathrm{dim}_k \,\Hom_{\cD_c(A^\op \otimes_k B)}(X,Y[n])\,;
\end{eqnarray}
\item[(iii)] whose composition law is induced by the (derived) tensor product of bimodules.
\end{itemize}
\begin{remark}
Note that sequence \eqref{eq:composed} follows from the natural $\otimes$-inclusions $\onil \subseteq \mathrm{Ker}(\overline{HP_\ast}) \subseteq \cN$.
\end{remark}
\subsection{Noncommutative mixed motives}\label{sub:mixed}
The rigid triangulated category $\Mix(k)_\bbQ$ of {\em noncommutative mixed motives} was introduced by Kontsevich in \cite{IAS}; see also\footnote{In {\em loc.~cit} we have considered non-connective algebraic $K$-theory. However, Kontsevich's original definition is in terms of {\em connective} algebraic $K$-theory.} \cite[\S9.2]{CT1}. It is defined in three steps: 
\begin{itemize}
\item[(i)] First, consider the category $\KPM(k)$ (enriched over spectra) whose objects are the saturated dg categories, whose morphisms from $\cA$ to $\cB$ are given by the (connective) algebraic $K$-spectrum $K(\cA^\op\otimes_k\cB)$, and whose composition law is induced by the (derived) tensor product of bimodules;
\item[(ii)] Then, take the formal triangulated envelope of $\KPM(k)$. Objects in this new category are formal finite extensions of formal shifts of objects in $\KPM(k)$. Let $\KTM(k)$ be the associated homotopy category; 
\item[(iii)] Finally, tensor each abelian group of morphisms of $\KTM(k)$ with $\bbQ$ and then pass to the associated pseudo-abelian envelope. The resulting category $\Mix(k)_\bbQ$ was named by Kontsevich as the category of noncommutative mixed motives (with rational coefficients).
\end{itemize}
\section{Proof of Theorem~\ref{thm:main1}}\label{sec:proof1}
Recall from \cite[\S4.1.3]{Andre} that an object of $\Chow(k)_F$ consists of a triple $(Z,e,r)$, where $Z$ is a smooth projective $k$-scheme, $e$ is an idempotent element of the $F$-algebra of correspondences $\cZ_{\mathrm{rat}}^{\mathrm{dim}(Z)} (Z\times Z)_F$ and $r \in \bbZ$. Moreover, we have
\begin{equation}\label{eq:morph-Chow}
\Hom_{\Chow(k)_F}((Z,e,r),(Z',e',r')):=e' \circ \cZ^{\mathrm{dim}(Z) -r+r'}_{\mathrm{rat}}(Z\times Z')_F \circ e\,.
\end{equation}
\begin{lemma}\label{lem:key1}
Let $\cV$ is a full subcategory of $\SmProj(k)$, which is stable under finite products and disjoint unions. If $\cV$ is contained in the category of finite {\'e}tale $k$-schemes, then we have the following identification 
$$ \Chow(k)_F^\cV=\{ M(Z):=(Z,\id_Z,0) \,|\, Z \in \cV\}^\natural \subset \Chow(k)_F\,.$$
\end{lemma}
\begin{proof}
Recall from \cite[\S4.1.6]{Andre} that by definition $\Chow(k)_F^\cV$ is the smallest subcategory of $\Chow(k)_F$ which contains the objects $M(Z), Z \in \cV$, and which is stable under direct sums, tensor products, direct factors, and duals. Hence, the inclusion $\Chow(k)_F^\cV \supseteq \{ M(Z)\,|\, Z \in \cV\}^\natural$ clearly holds. In order to prove the inclusion $\subseteq$ it suffices then to show that the category $\{M(Z)\,|\,Z \in \cV\}$ is already stable under direct sums, tensor products, and duals (the pseudo-abelian envelope preserves these constructions). In what concerns direct sums and tensor products this follows from the hypothesis that $\cV$ is stable under finite products and disjoint unions. In what concerns duals, this follows from the equality $\id_X=\id_X^t$ and from the fact that for every $Z \in \cV$ we have $\mathrm{dim}(Z)=0$; see \cite[\S4.1.4]{Andre}.
\end{proof}
Let $\bbQ(1) \in \Chow(k)_F$ be the Tate motive. Recall from \cite[Thm.~1.1]{CvsNC}\cite[Prop.~4.4]{Galois} that the functor \eqref{eq:functor-central} is given by the following composition
$$ \Chow(k)_F \stackrel{\pi}{\too} \Chow(k)_F\!/_{\!\!-\otimes \bbQ(1)} \stackrel{R}{\too} \NChow(k)_F\,,$$
where $\Chow(k)_F\!/_{\!\!-\otimes \bbQ(1)}$ denotes the orbit category associated to the auto-equivalence $-\otimes \bbQ(1)$. Both functors are faithful and symmetric monoidal, and $R$ is moreover full. Hence, in order to prove that the functor \eqref{eq:functor-central} is a (tensor) equivalence when $\cV$ is contained in the category of finite {\'e}tale $k$-schemes, it suffices by Lemma~\ref{lem:key1} to show that the functor $\pi$ becomes full when restricted to the objects $M(Z), Z \in \cV$. By \eqref{eq:morph-Chow} we have $\Hom_{\Chow(k)_F}(M(Z),M(Z')) = \cZ_{\mathrm{rat}}^0(Z\times Z')_F$. On the other hand, by definition of the orbit category, we have
$$ \Hom_{\Chow(k)_F\!/_{\!\!-\otimes \bbQ(1)}}(\pi M(Z),\pi M(Z')) = \bigoplus_{i \in \bbZ} \Hom_{\Chow(k)_F}(M(Z),M(Z')(i))\,.$$
Since by hypothesis the $k$-schemes $Z$ and $Z'$ are zero-dimensional, the $F$-vector spaces $\Hom_{\Chow(k)_F}(M(Z),M(Z')(i)) = \cZ_{\mathrm{rat}}^i(Z\times Z')_F$
are trivial for $i \neq 0$ and so we conclude that \eqref{eq:functor-central} is in fact a (tensor) equivalence.

Now, let $\cV$ be a (non-trivial) category which is {\em not} contained in the category of finite {\'e}tale $k$-schemes. There exists then a smooth projective $k$-scheme $Z$ of positive dimension $\mathrm{dim}(Z)$ belonging to $\cV$. As a consequence, we obtain the isomorphism
$$ \Hom_{\Chow(k)_F}(M(Z),M(Z)(-\mathrm{dim}(Z)))=\cZ_{\mathrm{rat}}^0(Z\times Z)_F \simeq F\,.$$
This implies that the homomorphism 
$$\End_{\Chow(k)_F}(M(Z)) \to\End_{\Chow(k)_F\!/_{\!\!-\otimes \bbQ(1)}}(\pi M(Z))$$
is {\em not} surjective and so we conclude that the functor \eqref{eq:functor-central} is not a (tensor) equivalence. 

\section{Proof of Theorem~\ref{thm:main2}}\label{sec:proof2}
By combining \cite[Thm.~1.12]{Semi}\cite[Prop.~4.4]{Galois} with the construction of the functor \eqref{eq:functor-central}, we obtain the following commutative diagram
\begin{equation}\label{eq:diagram-notation}
\xymatrix{
*+<2.5ex>{\AM(k)_F} \ar[rr]^-{\simeq} \ar@{^{(}->}[d] && *+<2.5ex>{\NAM(k)_F} \ar@{^{(}->}[d] \\
\Chow(k)_F \ar[d] \ar[r]^-\pi & \Chow(k)_F\!/_{\!\!-\otimes \bbQ(1)} \ar[d] \ar[r]^-R & \NChow(k)_F \ar[d] \\
\Num(k)_F \ar[r]_-\pi & \Num(k)_F\!/_{\!\!-\otimes \bbQ(1)} \ar[r]_-{R_\cN} & \NNum(k)_F\,.
}
\end{equation}
Since by Lemma~\ref{lem:key1} the category $\AM(k)_F$ consists only of the direct factors of zero-dimensional $k$-schemes, all the different equivalence relations on algebraic cycles are trivial. As a consequence, the composed functor
$$ \AM(k)_F \hookrightarrow \Chow(k)_F \too \Num(k)_F$$
is fully-faithful. Hence, an argument analogous to the one used in the proof of Theorem~\ref{thm:main1} shows us that the functor $\pi:\Num(k)_F \to \Num(k)_F\!/_{\!\!-\otimes \bbQ(1)}$ becomes fully-faithful when restricted to $\AM(k)_F$. Finally, the commutativity of the above diagram combined with the fact that $R_\cN$ is fully-faithful, allow us to conclude that the composed functor \eqref{eq:composed} becomes fully-faithful when restricted to $\NAM(k)_F$. 

\section{Proof of Theorem~\ref{thm:main3}}\label{sec:proof3}
We start with item (i). Recall from \S\ref{sub:homological} that since $F$ is a field extension of $k$ we have a well-defined $F$-linear $\otimes$-functor
$$ \overline{HP_\ast}:\NNum(k)_F \too \mathrm{sVect}(F)\,.$$
Given $M_1, M_2 \in \NNum(k)_F$, recall from \cite[Prop.~9.7]{Galois} that the  symmetry isomorphism constraints of the category $\NNum^\dagger(k)_F$ are given by $c^\dagger_{M_1,M_2}:= c_{M_1,M_2} \circ (e_{M_1} \otimes e_{M_2})$, where $e_M$ is the endomorphism $2\cdot \underline{\pi}^+_M - \id_M$ of $M$ with $\underline{\pi}^+_M$ the noncommutative correspondence associated to the K{\"u}nneth projector $\pi^+_M: \overline{HP_\ast}(M) \twoheadrightarrow \overline{HP_\ast}^+(M) \hookrightarrow \overline{HP_\ast}(M)$. By combining Lemma~\ref{lem:key1} with the $\otimes$-equivalence $\AM(k)_F\simeq \NAM(k)_F$ we obtain then the following identification
\begin{equation}\label{eq:description}
\NAM(k)_F = \{\cU(\cD_\perf^\dg(Z))\,|\, Z \,\textrm{finite {\'e}tale $k$-scheme} \}^\natural \subset \NChow(k)_F\,.
\end{equation}
Hence, in order to prove item (i) it suffices to show that for every finite {\'e}tale $k$-scheme $Z$ the noncommutative correspondence $\underline{\pi}^+_{\cU(\cD_\perf^\dg(Z))}$ agrees with $\id_{\cU(\cD_\perf^\dg(Z))}$. As explained in the proof of \cite[Thm.~1.3]{Galois}, $\overline{HP_\ast}(\cU(\cD_\perf^\dg(Z)))$ identifies with $HP_\ast(\cD_\perf^\dg(Z))$ and 
\begin{eqnarray*}
HP_\ast^+(\cD_\perf^\dg(Z))\simeq \bigoplus_{n\,\textrm{even}}H^n_{dR}(Z) && HP_\ast^-(\cD_\perf^\dg(Z))\simeq \bigoplus_{n\,\textrm{odd}}H^n_{dR}(Z)\,,
\end{eqnarray*}
where $H_{dR}(-)$ denotes de Rham cohomology.
Since $Z$ is zero-dimensional, $H^n_{dR}(Z)=0$ for $n\neq 0$ and so the equality $\underline{\pi}^+_{\cU(\cD_\perf^\dg(Z))}=\id_{\cU(\cD_\perf^\dg(Z))}$ holds.

Item (ii) follows from the Tannakian formalism; see \cite[\S2.3.3]{Andre}. Let us now prove item (iii). As explained in \cite[\S1.3]{Andre}, the classical Galois-Grothendieck correspondence between finite {\'e}tale $k$-schemes and finite $\mathrm{Gal}(\overline{k}/k)$-sets admits a $F$-linearization making the following diagram commutative
$$
\xymatrix{
\{\textrm{finite {\'e}tale $k$-schemes}\}^\op \ar[d]_-{M} \ar[r]_{\simeq}^-{Z \mapsto Z(\overline{k})} & \{\textrm{finite $\mathrm{Gal}(\overline{k}/k)$-sets} \}^\op  \ar[d]^-{S \mapsto F^S}\\
\AM(k)_F \ar[r]_-{\simeq} & \{\textrm{finite dim. $F$-linear representations of $\mathrm{Gal}(\overline{k}/k)$} \} \,.
}
$$
Recall from \cite[\S1.3.1]{Andre} that $Z=\mathrm{spec}(A)$, with $A$ a finite dimensional {\'e}tale $k$-algebra. As explained in \cite[Remark~6.18]{DelMil}, we have a natural isomorphism $H^0_{dR}(Z)=H_{dR}(Z)\simeq A$ of $k$-vector spaces. Since $A \otimes_k \overline{k} \simeq \overline{k}^{Z(\overline{k})}$, the $k$-algebra $A$ is of dimension $Z(\overline{k})$ and so $H^0_{dR}(Z)$ also identifies with $k^{Z(\overline{k})}$. Hence, via the $\otimes$-equivalence $\AM(k)_F \simeq \NAM(k)_F$, the fiber functor
\begin{eqnarray*}
\overline{HP_\ast}: \NAM(k)_F \too \mathrm{Vect}(F) && \cU(\cD_\perf^\dg(Z)) \mapsto k^{Z(\overline{k})}\otimes_k F
\end{eqnarray*}
corresponds to the fiber functor
\begin{eqnarray}\label{eq:fiber-functor}
\AM(k)_F \too \mathrm{Vect}(F) && Z \mapsto F^{Z(\overline{k})}\,.
\end{eqnarray}
The forgetful functor from the category of finite dimensional $F$-linear representations of $\mathrm{Gal}(\overline{k}/k)$ to $\mathrm{Vect}(F)$ identifies with \eqref{eq:fiber-functor}, and so one concludes that the noncommutative motivic Galois group $\mathrm{Gal}(\NAM(k)_F)\simeq \mathrm{Gal}(\AM(k)_F)$ is naturally isomorphic to the absolute Galois group of $k$.
\section{Base-change}\label{sec:basechange}
In this section we first develop the base-change mechanism in the world of noncommutative pure motives (see Theorems \ref{thm:basechange1} and \ref{thm:basechange2}), then prove its compatibility with the classical base-change mechanism (see Theorem~\ref{thm:compatibility}), and finally apply it to the study of motivic decompositions and Schur/Kimura finiteness (see \S\ref{sub:decompositions}-\ref{sub:Schur}).
\begin{theorem}\label{thm:basechange1}
Every field extension $k'/k$ gives rise to $F$-linear $\otimes$-functors 
\begin{eqnarray*}
-\otimes_kk':\NChow(k)_F \too \NChow(k')_F && -\otimes_kk': \NNum(k)_F \too \NNum(k')_F\,.
\end{eqnarray*}
Assuming the noncommutative standard conjecture $C_{NC}$, that $k$ is of characteristic zero, and that $k$ is a field extension of $F$ or that $F$ is a field extension of $k'$, we have moreover a well-defined $F$-linear $\otimes$-functor $-\otimes_kk':\NNum^\dagger(k)_F \to \NNum^\dagger(k')_F$. 
\end{theorem}

\begin{theorem}\label{thm:basechange2}
When the field extension $k'/k$ is finite and separable we have the following adjunctions
\begin{equation}\label{eq:adjunctions}
\xymatrix{
\NChow(k')_F \ar@<1ex>[d]^-{(-)_k}   &&  \NNum(k')_F \ar@<1ex>[d]^-{(-)_k}  \\
\NChow(k)_F \ar@<1ex>[u]^-{-\otimes_kk'}  && \NNum(k)_F \ar@<1ex>[u]^-{-\otimes_kk'}\,,
}
\end{equation}
where $(-)_k$ denotes the restriction functor.
\end{theorem}
The remaining of this section is devoted to the proof of Theorems~\ref{thm:basechange1} and \ref{thm:basechange2}.
\subsection{Proof of Theorem~\ref{thm:basechange1}}\label{sub:basechange1}
The classical functor $-\otimes_kk': \cC(k) \to \cC(k')$ between complexes of vector spaces is symmetric monoidal and hence gives naturally rise to a $\otimes$-functor $-\otimes_kk': \dgcat(k) \to \dgcat(k')$.
\begin{proposition}\label{prop:Morita}
The functor $-\otimes_kk':\dgcat(k) \to \dgcat(k')$ preserves derived Morita equivalences.
\end{proposition}
\begin{proof}
Let us start by proving that the functor $-\otimes_kk'$ preserves quasi-equivalences. Let $G: \cA \to \cB$ be a quasi-equivalence in $\dgcat(k)$. Since the functor $-\otimes_kk':\cC(k) \to \cC(k')$ preserves quasi-isomorphisms the dg functor $G\otimes_kk'$ satisfies condition (i) of Definition~\ref{def:quasi-eq}. In what concerns condition (ii) we have a natural (K{\"u}nneth) isomorphism $\dgHo(G)\otimes_kk' \simeq \dgHo(G\otimes_kk')$ of $k'$-linear functors. Hence, since by hypothesis $\dgHo(G)$ is essentially surjective we conclude that $\dgHo(G\otimes_kk')$ is also essentially surjective.  Now, let $G:\cA \to \cB$ be a derived Morita equivalence. Note that we have a commutative diagram
$$
\xymatrix{
*+<2.5ex>{\cA} \ar[d]_-h \ar[r]^-G & *+<2.5ex>{\cB} \ar[d]^-h \\
\widehat{\cA} \ar[r]_-{\widehat{G}} & \widehat{\cB}\,,
}
$$ 
where $h$ stands for the Yoneda dg functor; see Notation \ref{not:Yoneda}. By Proposition~\ref{prop:char}(ii) $\widehat{G}$ is a quasi-equivalence. Since the functor $-\otimes_kk': \dgcat(k) \to \dgcat(k')$ preserves quasi-equivalences we conclude that $\widehat{G}\otimes_kk'$ is also a quasi-equivalence and hence a derived Morita equivalence. By the two-out-of-three property of derived Morita equivalences it suffices then to show that the dg functor $h \otimes_kk': \cA \otimes_kk' \to \widehat{\cA}\otimes_kk'$ is a derived Morita equivalence for every dg category $\cA$. The same argument as above shows that $h\otimes_kk'$ satisfies condition (i) of Definition~\ref{def:quasi-eq}. By Proposition~\ref{prop:char}(iii) it remains only to show that every object $X\in \dgHo(\widehat{\cA}\otimes_kk') \subset \cD_c(\widehat{\cA}\otimes_kk')$ is in the essential image of the functor $\bbL(h\otimes_kk')_!: \cD_c(\cA\otimes_kk') \to \cD_c(\widehat{\cA}\otimes_kk')$. The natural isomorphism $\dgHo(\widehat{\cA})\otimes_kk' \simeq \dgHo(\widehat{\cA}\otimes_kk')$ of $k'$-linear categories allows us to write $X$ as $\overline{X}\otimes_kk'$, where $\overline{X}$ is the corresponding object of $\dgHo(\widehat{\cA})$. Moreover, since $h:\cA \hookrightarrow \widehat{\cA}$ is a derived Morita equivalence, there exists by Proposition~\ref{prop:char}(iii) an object $\overline{Y} \in \cD_c(\cA)$ and an isomorphism $\bbL h_!(\overline{Y}) \simeq \overline{X}$ in $\cD_c(\widehat{\cA})$. Now, note that we have the following commutative diagram
$$
\xymatrix{
\cD_c(\cA) \ar[d]_-{-\otimes_kk'} \ar[rr]^-{\bbL h_!} && \cD_c(\widehat{\cA}) \ar[d]^-{-\otimes_k k'} \,.\\
\cD_c(\cA\otimes_kk') \ar[rr]_-{\bbL(h\otimes_kk')_!} && \cD_c(\widehat{\cA}\otimes_kk')\,.
}
$$
By taking $Y:=\overline{Y}\otimes_kk'$ we then obtain
$$ \bbL(h\otimes_kk')_!(Y)\simeq \bbL h_!(\overline{Y})\otimes_kk' \simeq \overline{X}\otimes_kk' \simeq X$$
and so the proof is finished.
\end{proof}

Given dg categories $\cA$ and $\cB$ and a $\cA$-$\cB$-bimodule $X$, let us denote by $T(\cA,\cB;X)$ the dg category whose set of objects is the disjoint union of the sets of objects of $\cA$ and $\cB$ and whose morphisms are given by: $\cA(x,y)$ if $x,y \in \cA$; $\cB(x,y)$ if $x,y \in \cB$; $X(x,y)$ if $x \in \cA$ and $y \in \cB$; $0$ if $x \in \cB$ and $y \in \cA$. Composition is induced by the composition on $\cA$ and $\cB$, and by the $\cA$-$\cB$-bimodule structure of $X$. Note that we have two natural inclusion dg functors $\iota_\cA:\cA\to T(\cA,\cB;X)$ and $\iota_\cB:\cB \to T(\cA,\cB;X)$.

\begin{lemma}\label{lem:key2}
We have a natural identification 
$$ T(\cA,\cB;X)\otimes_kk' \simeq T(\cA\otimes_kk',\cB\otimes_kk';X\otimes_kk')$$
of dg $k'$-linear categories.
\end{lemma}
\begin{proof}
This follows from the fact the the functor $-\otimes_kk'$ does not alter the set of objects and from the fact that $X\otimes_kk'$ is naturally a $(\cA\otimes_kk')$-$(\cB\otimes_kk')$-bimodule.
\end{proof}
Now, recall from \cite[\S4]{survey} that a functor $E:\dgcat \to \mathsf{A}$, with values in an additive category, is called an {\em additive invariant} if it sends derived Morita equivalences to isomorphisms, and for every $\cA$, $\cB$ and $X$ (as above) the inclusion dg functors $\iota_\cA$ and $\iota_\cB$ give rise to an isomorphism
$$ E(\cA) \oplus E(\cB) \stackrel{\sim}{\too} E(T(\cA,\cB;X))\,.$$
Examples of additive invariants include algebraic $K$-theory, cyclic homology (and all its variants), topological Hochschild homology, etc. In \cite{IMRN} the universal additive invariant was constructed. It can be described as follows: let $\Hmo_0$ be the category whose objects are the dg categories and whose morphisms are given by $\Hom_{\Hmo_0}(\cA,\cB):=K_0\rep(\cA,\cB)$, where $\rep(\cA,\cB) \subset \cD(\cA^\op \otimes \cB)$ is the full triangulated subcategory of those $\cA\text{-}\cB$-bimodules $X$ such that $X(-,x) \in \cD_c(\cB)$ for every object $x\in \cA$. The composition law is induced by the tensor product of bimodules. We have a natural functor
\begin{equation}\label{eq:universal}
\dgcat \too \Hmo_0
\end{equation}
which is the identity on objects and which maps a dg functor to the class (in the Grothendieck group) of the naturally associated bimodule. The category $\Hmo_0$ is additive and the functor \eqref{eq:universal} is an additive invariant. Moreover, it is characterized by the following universal property:
\begin{theorem}[see \cite{IMRN}]\label{thm:universal}
Given an additive category $\mathsf{A}$, the above functor \eqref{eq:universal} induces an equivalence of categories
$$ \Fun_{\mathsf{add}}(\Hmo_0,\mathsf{A}) \stackrel{\sim}{\too} \Fun_{\mathsf{add\, inv}}(\dgcat,\mathsf{A})\,,$$
where the left hand-side denotes the category of additive functors and the right hand-side the category of additive invariants.
\end{theorem}
By combining Theorem~\ref{thm:universal} with Proposition \ref{prop:Morita} and Lemma \ref{lem:key2} we obtain then a well-defined additive functor
\begin{equation}\label{eq:functor-Hmo}
-\otimes_kk': \Hmo_0(k) \too \Hmo_0(k')\,,
\end{equation}
which is moreover symmetric monoidal.
\begin{proposition}\label{prop:induced}
The functor \eqref{eq:functor-Hmo} gives rise to a $F$-linear $\otimes$-functor
\begin{equation}\label{eq:functor1}
 -\otimes_kk':\NChow(k)_F \too \NChow(k')_F\,.
\end{equation}
\end{proposition}
\begin{proof}
Recall from \S\ref{sub:saturated} that the saturated dg categories are precisely the dualizable (or rigid) objects of the homotopy category $\Hmo_0$ of dg categories. Hence, since the functor $-\otimes_kk':\dgcat(k) \to \dgcat(k')$ is symmetric monoidal (and by Proposition~\ref{prop:Morita} preserves derived Morita equivalences) we conclude that it preserves saturated dg categories. Therefore, if we denote by $\sHmo_0 \subset \Hmo_0$ the full subcategory of saturated dg categories, \eqref{eq:functor-Hmo} restricts to a $\otimes$-functor $-\otimes_kk':\sHmo_0(k) \to \sHmo(k')$. Given saturated dg categories $\cA$ and $\cB$, we have a natural equivalence of triangulated categories $\rep(\cA,\cB)\simeq \cD_c(\cA^\op \otimes\cB)$ and hence a group isomorphism $K_0\rep(\cA,\cB)\simeq K_0(\cA^\op\otimes\cB)$. The categories $\NChow(k)_F$ and $\NChow(k')_F$ can then be obtained from $\sHmo_0(k)$ and $\sHmo_0(k')$ by first tensoring each abelian group of morphisms with the field $F$ and then passing to the associated pseudo-abelian envelope. As a consequence, the $\otimes$-functor $-\otimes_kk':\sHmo_0(k) \to \sHmo_0(k')$ gives rise to the $\otimes$-functor \eqref{eq:functor1}.
\end{proof}
Now, recall from \S\ref{sub:numerical} the description of the category $\NNum(k)_F$ of noncommutative numerical motives which uses saturated dg algebras and the bilinear form $\chi$. Recall also that every saturated dg category is derived Morita equivalent to a saturated dg algebra. The category $\NNum(k)_F$ can then be obtained from $\sHmo_0(k)$ (see the proof of Proposition~\ref{prop:induced}) by first quotienting out by the $\otimes$-ideal $\mathrm{Ker}(\chi)$, then tensoring each abelian group of morphisms with the field $F$, and finally passing to the associated pseudo-abelian envelope. In order to obtain from \eqref{eq:functor1} the $F$-linear $\otimes$-functor $-\otimes_kk':\NNum(k)_F \to \NNum(k')_F$, it suffices then to show that the functor $-\otimes_kk':\sHmo_0(k) \to \sHmo_0(k')$ preserves the $\otimes$-ideal $\mathrm{Ker}(\chi)$, \ie given any two saturated dg $k$-algebras $A$ and $B$ one needs to show that the homomorphism
\begin{equation}\label{eq:morphism}
\xymatrix@C=3em@R=.7em{
\Hom_{\sHmo_0(k)}(A,B) \ar@{=}[d] \ar[r] & \Hom_{\sHmo_0(k')}(A\otimes_k k', B \otimes_k k')  \ar@{=}[d] \\
K_0(A^\op\otimes_k B) \ar[r] & K_0(A^\op\otimes_kB\otimes_kk')
}
\end{equation}
preserves the kernel of the bilinear form $\chi$. Since $K_0(A^\op\otimes_kB)$ is generated by the elements of shape $[X]$, with $X \in \cD_c(A^\op \otimes_kB)$, and similarly $K_0(A^\op\otimes_kB\otimes_kk')$ is generated by the elements $[Y]$, with $Y \in \cD_c(A^\op\otimes_kB\otimes_kk')$, it suffices to prove the following implication:
\begin{equation}\label{eq:implication}
\textrm{
$[X] \in \mathrm{Ker}(\chi) \Rightarrow \chi([X\otimes_kk'],[Y])=0 \,\, \textrm{for all} \,\,Y \in \cD_c(A^\op\otimes_kB\otimes_kk')$.
}
\end{equation}
The proof of this implication is divided into a finite and an infinite case. Let us assume first that the field extension $k'/k$ is finite.
\begin{lemma}\label{lem:compact}
Given a saturated dg $k$-algebra $C$, we have an adjunction
\begin{equation}\label{eq:adjunction}
\xymatrix{
\cD_c(C\otimes_kk') \ar@<1ex>[d]^-R \\
\cD_c(C) \ar@<1ex>[u]^-{-\otimes_kk'}\,.
}
\end{equation}
Moreover, the equality $\chi([X],[R(Y)])=[k':k]\cdot \chi([X\otimes_kk'], [Y])$ holds for every $X \in \cD_c(C)$ and $Y \in \cD_c(C\otimes_kk')$.
\end{lemma}
\begin{proof}
Consider the natural adjunction
\begin{equation}\label{eq:adjunction1}
\xymatrix{
 \cD(C\otimes_kk') \ar@<1ex>[d]^-R \\
  \cD(C) \ar@<1ex>[u]^-{-\otimes_kk'}\,,
}
\end{equation}
where $R$ denotes the functor obtained by composing the equivalence $\cD(C\otimes_kk')\simeq \cD((C\otimes_kk')_k)$ with the restriction along the canonical map $C \to (C\otimes_kk')_k$. The fact that the functor $-\otimes_kk'$ restricts to compact objects is clear. In what concerns $R$, this follows from the isomorphism 
$$R(C\otimes_kk') \simeq \underbrace{C \oplus \cdots \oplus C}_{[k':k]\text{-}\mathrm{times}}$$
 in $\cD(C)$ and from the fact that the field extension $k'/k$ is finite. The adjunction \eqref{eq:adjunction} is then obtained by restricting \eqref{eq:adjunction1} to compact objects. By adjunction we have a natural isomorphism of $k$-vector spaces
$$ \Hom_{\cD_c(C)}(X,R(Y)[n]) \simeq \Hom_{\cD_c(C\otimes_kk')}(X\otimes_kk',Y[n])$$
for every $X \in \cD_c(C)$, $Y \in \cD_c(C\otimes_kk')$ and $n\in \bbZ$. Moreover, the equality holds
$$\mathrm{dim}_k \Hom_{\cD_c(C\otimes_kk')}(X\otimes_kk',Y[n]) =  [k':k]\cdot \mathrm{dim}_{k'} \Hom_{\cD_c(C\otimes_kk')}(X\otimes_kk',Y[n])\,.$$
By definition of the bilinear form \eqref{eq:bilinear} we then obtain the searched equality.
\end{proof}
By applying Lemma~\ref{lem:compact} to the saturated dg $k$-algebra $C:=A^\op\otimes_kB$, one obtains then the equality $\chi([X],[R(Y)])=[k':k]\cdot \chi([X\otimes_kk'], [Y])$ for every $X \in \cD_c(A^\op\otimes_kB)$ and $Y \in \cD_c(A^\op\otimes_kB\otimes_kk')$. Hence, since $[k':k]\neq 0$ one concludes that the above implication \eqref{eq:implication} holds.

Let us now assume that $k'/k$ is an infinite field extension. Note that $k'$ identifies with the colimit of the filtrant diagram $\{k_i\}_{i \in I}$ of intermediate field extensions $k'/k_i/k$ which are finite over $k$. Given a saturated dg $k$-algebra $C$ one obtains then (by applying the usual strictification procedure) a filtrant diagram $\{\cD_c(C\otimes_kk_i)\}_{i \in I}$ of triangulated categories. Let us denote by $\mathrm{colim}_{i\in I} \cD_c(C\otimes_kk_i)$ its colimit and by $\iota_i: \cD_c(C\otimes_kk_i) \to \mathrm{colim}_{i\in I}\cD_c(C\otimes_kk_i)$ the corresponding functors. The functors $-\otimes_{k_i}k': \cD_c(C\otimes_kk_i) \to \cD_c(C\otimes_kk'), i \in I$, give rise to a well-defined functor
\begin{equation*}
\Gamma: \mathrm{colim}_{i \in I} \cD_c(C\otimes_kk_i) \too \cD_c(C\otimes_kk')\,.
\end{equation*}
\begin{proposition}\label{prop:key} 
The functor $\Gamma$ is an equivalence. Moreover, given objects $X$ and $Y$ in $\cD_c(C\otimes_kk')$, there exists an element $j \in I$ and objects $X_j, Y_j \in \cD_c(C\otimes_kk_j)$ such that $X_j \otimes_{k_j}k'\simeq X, Y_j \otimes_{k_j}k' \simeq Y$ and $\chi([X_j],[Y_j])=\chi([X],[Y])$.
\end{proposition}
\begin{proof}
We start by proving that $\Gamma$ is fully-faithful. Let $O$ and $P$ be two objects of $\mathrm{colim}_{i \in I} \cD_c(C\otimes_kk_i)$. By construction, there is an element $j \in I$ and objects $O_j, P_j \in \cD_c(C\otimes_kk_j)$ such that $\iota_j(O_j)=O$ and $\iota_j(P_j)=P$. Moreover, we have the following equality
$$ \Hom_{\mathrm{colim}_{i \in I} \cD_c(C\otimes_kk_i)}(O,P) = \mathrm{colim}_{j \downarrow I} \Hom_{\cD_c(C\otimes_kk_i)}(O_j\otimes_{k_j}k_i,P_j\otimes_{k_j}k_i)\,,$$
where $j \downarrow I:=\{j \to i | i \in I\}$ denotes the category of objects under $j$. On the other hand we have the equality
$$ \Hom_{\cD_c(C\otimes_kk')}(\Gamma(O),\Gamma(P))=\Hom_{\cD_c(C\otimes_kk')}(O_j\otimes_{k_j}k',P_j\otimes_{k_j}k')\,.$$
Now, note that by adjunction $\mathrm{colim}_{j \downarrow I} \Hom_{\cD_c(C\otimes_kk_i)}(O_j\otimes_{k_j}k_i,P_j\otimes_{k_j}k_i)$ identifies with $\mathrm{colim}_{j \downarrow I} \Hom_{\cD_c(C\otimes_kk_j)}(O_j,P_j\otimes_{k_j}k_i)$. Moreover, since $O_j$ is a compact object this colimit agrees with $\Hom_{\cD(C\otimes_kk_j)}(O_j,\mathrm{colim}_{j \downarrow I}P_j\otimes_{k_j}k_i)$. Similarly, we have 
$$ \Hom_{\cD_c(C\otimes_kk')}(\Gamma(O),\Gamma(P))\simeq\Hom_{\cD(C\otimes_kk_j)}(O_j,P_j\otimes_{k_j}k')\,.$$
Hence, it suffices to show that the induced morphism 
$$\gamma(P_j): \mathrm{colim}_{j \downarrow I}\, P_j \otimes_{k_j}k_i \to P_j \otimes_{k_j}k'$$
is an isomorphism. Observe that we have a natural transformation
$$ \gamma(-): \mathrm{colim}_{j \downarrow I}\, - \otimes_{k_j}k_i \Rightarrow - \otimes_{k_j}k'$$
between two triangulated endofunctors of $\cD(C\otimes_kk_j)$. Since $C\otimes_kk_j$ generates the triangulated category $\cD(C\otimes_kk_j)$ and $P_j$ is a compact object it suffices then to show that $\gamma(C\otimes_kk_j)$ is an isomorphism. This is the case since this morphism identifies with the natural isomorphism $\mathrm{colim}_{j \downarrow I} \, C\otimes_kk_i \stackrel{\sim}{\to} C\otimes_kk'$. 

Let us now prove that $\Gamma$ is essentially surjective. By construction, every object of $\cD_c(C\otimes_kk')$ is obtained from $C\otimes_kk'$ by a finite number of (de)suspensions, extensions and retracts. Since the functors $-\otimes_{k_i}k':\cD_c(C\otimes_kk_i) \to \cD_c(C\otimes_kk'), i \in I$, are triangulated and $\Gamma$ is fully-faithful, it suffices then to show that $C\otimes_kk'$ is in the (essential) image of $\Gamma$. This is clearly the case and so the proof that $\Gamma$ is an equivalence is finished.

Now, let $X$ and $Y$ be two objects of $\cD_c(C\otimes_kk')$. Since $\Gamma$ is an equivalence there exists an element $j \in I$ and objects $X_j,Y_j \in \cD_c(C\otimes_kk_j)$ such that $X_j \otimes_{k_j}k' \simeq X$ and $Y_j \otimes_{k_j}k'\simeq Y$. Moreover, we have the following isomorphism
$$ \Hom_{\cD_c(C\otimes_kk')}(X,Y[n]) \simeq \mathrm{colim}_{j \downarrow I}\Hom_{\cD_c(C\otimes_{k_j}k_i)}(X_j \otimes_{k_j}k_i, (Y_j \otimes_{k_j}k_i)[n])$$
for every $n \in \bbZ$. Since $k'$ is the colimit of the diagram $\{k_i\}_{i \in I}$ and
$$\Hom_{\cD_c(C\otimes_{k_j}k_i)}(X_j \otimes_{k_j}k_i, (Y_j \otimes_{k_j}k_i)[n]) \simeq \Hom_{\cD_c(C\otimes_kk_j)}(X_j,Y_j[n]) \otimes_{k_j}k_i$$
for every $n \in \bbZ$, we conclude then that 
$$\Hom_{\cD_c(C\otimes_kk')}(X,Y[n]) \simeq \Hom_{\cD_c(C\otimes_kk_j)}(X_j, Y_j[n])\otimes_{k_j}k'\,.$$
As a consequence, the dimension $ \mathrm{dim}_{k'}\Hom_{\cD_c(C\otimes_kk')}(X,Y[n])$ agrees with the dimension $\mathrm{dim}_{k_j}\Hom_{\cD_c(C\otimes_kk_j)}(X_j,Y_j[n])$ and so by definition of the bilinear form \eqref{eq:bilinear} we obtain the equality $\chi([X],[Y])=\chi([X_j],[Y_j])$.
\end{proof}
By combining Proposition~\ref{prop:key} (applied to $C=A^\op \otimes_k B$) with Lemma~\ref{lem:compact} (applied to $k'=k_j$ and $C=A^\op \otimes_k B$) we obtain the following equalities
$$ \chi([X],[R(Y_j)])=[k_j:k]\cdot \chi([X\otimes_kk_j],[Y_j])=[k_j:k] \cdot \chi([X\otimes_kk'],[Y])\,.$$
Since $[k_j:k]\neq 0$ one concludes then that the above implication \eqref{eq:implication} holds and so we obtain the $F$-linear $\otimes$-functor $-\otimes_kk': \NNum(k)_F \to \NNum(k')_F$.

Now, recall from \S\ref{sub:homological} that periodic cyclic homology $HP$ gives rise to a well-defined $F$-linear $\otimes$-functor $\overline{HP_\ast}: \NChow(k)_F \to \sVect(K)$.
\begin{proposition}\label{prop:HP}
The following diagrams commute (up to natural isomorphism)
\begin{equation}\label{eq:squares}
\xymatrix{
\NChow(k)_F \ar[d]_-{\overline{HP_\ast}} \ar[r]^-{-\otimes_kk'} & \NChow(k')_F \ar[d]^{\overline{HP_\ast}} & \NChow(k)_F \ar[dr]_-{\overline{HP_\ast}} \ar[r]^-{-\otimes_kk'} &  \NChow(k')_F \ar[d]^-{\overline{HP_\ast}}  \\
\mathrm{sVect}(k) \ar[r]_-{-\otimes_kk'}& \mathrm{sVect}(k')& &\mathrm{sVect}(F)\,. \\
}
\end{equation}
The left hand-side diagram concerns the case where $k$ is a field extension of $F$ and the right hand-side diagram the case where $F$ is a field extension of $k'$.
\end{proposition}
\begin{proof}
Given a dg category $\cA$, we have a natural isomorphism $CC(\cA)\otimes_kk' \simeq CC(\cA\otimes_kk')$ of $k'$-linear cyclic bicomplexes; see \cite[\S2.1.2]{Loday}. As a consequence, one obtains natural isomorphisms
$$ HC_n(\cA) \otimes_kk' \stackrel{\sim}{\too} HC_n(\cA\otimes_kk') \qquad n \in \bbZ\,,$$
where $HC_n$ denotes the $n^{\mathrm{th}}$ cyclic homology group. These isomorphisms are compatible with Connes' periodicity operator $S$ (see \cite[\S2.2]{Loday}) and hence give rise to a well-defined commutative diagram
$$
\xymatrix@C=2em@R=2em{
\cdots \ar[r]^-{S\otimes_kk'} & HC_{n+2r}(\cA)\otimes_kk' \ar[d]^\simeq  \ar[r]^-{S\otimes_kk'} & HC_{n+2r-2}(\cA) \otimes_kk' \ar[d]^-\simeq \ar[r]^-{S\otimes_kk'} &\cdots  \ar[r]^-{S\otimes_kk'}  \ar[r]^-{S\otimes_kk'}  & HC_n(\cA) \otimes_kk' \ar[d]^-\simeq \\
\cdots \ar[r]_-{S} & HC_{n+2r}(\cA\otimes_kk') \ar[r]_-{S} & HC_{n+2r-2}(\cA \otimes_kk') \ar[r]_-{S}  \ar[r]_-{S} & \cdots \ar[r]_-{S}& HC_n(\cA \otimes_kk')\,.
}
$$
By passing to the limit we obtain then induced isomorphisms
\begin{equation}\label{eq:isom1}
\underset{r}{\mathrm{lim}}(HC_{n+2r}(\cA) \otimes_kk') \stackrel{\sim}{\too} \underset{r}{\mathrm{lim}}\, HC_{n+2r}(\cA\otimes_kk') \qquad n \in \bbZ\,.
\end{equation}
When $\cA$ is saturated, the periodicity operator $S$ becomes an isomorphism for $|r| \gg 0$ and the $n^{\mathrm{th}}$ periodic cyclic homology group $HP_n(\cA)$ identifies with $\mathrm{lim}_rHC_{n+2r}(\cA)$; see the proof of \cite[Prop.~7.9]{Galois}. As a consequence, \eqref{eq:isom1} corresponds to an isomorphism $HP_n(\cA) \otimes_kk' \stackrel{\sim}{\to} HP_n(\cA\otimes_kk')$. This allows us to conclude that the diagram on the left hand-side of \eqref{eq:squares} is commutative. The commutativity of the diagram on the right hand-side is proved similarly: simply use moreover the commutative diagram
$$
\xymatrix{
\mathrm{sVect}(k) \ar[dr]_-{-\otimes_kF} \ar[r]^-{-\otimes_kk'} & \mathrm{sVect}(k') \ar[d]^-{-\otimes_{k'}F} \\
& \mathrm{sVect}(F)\,.
}
$$
\end{proof}
Assuming the noncommutative standard conjecture $C_{NC}$, that $k$ (and hence $k'$) is of characteristic zero, and that $k$ is a field extension of $F$ or that $F$ is a field extension of $k'$, recall from the proof of Theorem~\ref{thm:main3} the construction of the symmetric monoidal categories $\NNum^\dagger(k)_F$ and $\NNum^\dagger(k')_F$. They are obtained by applying the general \cite[Prop.~9.7]{Galois} to the $F$-linear $\otimes$-functors
\begin{eqnarray*}
\overline{HP_\ast}: \NChow(k)_F \too \mathrm{sVect}(F) && \NChow(k)_F \to \NNum(k)_F \\
\overline{HP_\ast}: \NChow(k')_F \too \mathrm{sVect}(F) && \NChow(k')_F \to \NNum(k')_F \,.
\end{eqnarray*}
Thanks to Proposition~\ref{prop:HP}, the functor $-\otimes_kk':\NChow(k)_F \to \NChow(k')_F$ preserves these new symmetry isomorphism constraints and hence gives rise to a well-defined $F$-linear $\otimes$-functor $-\otimes_kk':\NNum^\dagger(k)_F \to \NNum^\dagger(k')_F$. This concludes the proof of Theorem~\ref{thm:basechange1}.
\subsection{Proof of Theorem~\ref{thm:basechange2}}\label{sub:basechange2}
The field extension $k'/k$ gives rise to an adjunction
\begin{equation}\label{eq:adj}
\xymatrix{
\dgcat(k') \ar@<1ex>[d]^-{(-)_k} \\
\dgcat(k) \ar@<1ex>[u]^{-\otimes_kk'}\,.
}
\end{equation}
\begin{lemma}\label{lemma:key2}
The functor $(-)_k: \dgcat(k') \to \dgcat(k)$ preserves derived Morita equivalences and we have a natural identification $T(\cA,\cB;X)_k \simeq T(\cA_k,\cB_k;X_k)$ of dg $k$-categories for all dg $k'$-categories $\cA$ and $\cB$ and $\cA\text{-}\cB$-bimodules~$X$. 
\end{lemma}
\begin{proof}
The first claim follows from the fact that for every $\cA \in \dgcat(k')$ we have a natural equivalence $\cD(\cA) \simeq \cD(\cA_k)$ of $k$-categories. The second claim follows from the fact that the functor $(-)_k$ does not alter the set of objects and that $X_k$ is naturally a $\cA_k\text{-}\cB_k$-bimodule.
\end{proof}
Lemma~\ref{lemma:key2} combined with the universal property of \eqref{eq:universal} (see Theorem~\ref{thm:universal}) furnish us a well-defined additive functor $(-)_k:\Hmo_0(k') \to \Hmo_0(k)$, which is moreover symmetric monoidal. Let $\cA \in \Hmo_0(k)$ and $\cB \in \Hmo_0(k')$. The natural equivalence of categories $\cD((\cA\otimes_kk')^\op \otimes_{k'}\cB) \simeq \cD(\cA^\op \otimes_k \cB_k)$ restricts to an equivalence $\rep(\cA\otimes_kk', \cB) \simeq \rep(\cA,\cB_k)$ and hence gives rise to an isomorphism
$$ \Hom_{\Hmo_0(k')}(\cA\otimes_kk', \cB)=K_0\rep(\cA\otimes_kk',\cB) \simeq K_0\rep(\cA,\cB_k) = \Hom_{\Hmo_0(k)}(\cA,\cB_k)\,.$$
As a consequence, one obtains the following adjunction
\begin{equation}\label{eq:adju}
\xymatrix{
\Hmo_0(k') \ar@<1ex>[d]^-{(-)_k} \\
\Hmo_0(k) \ar@<1ex>[u]^{-\otimes_kk'}\,.
}
\end{equation}
\begin{lemma}\label{lemma:key}
Given dg $k'$-algebras $A$ and $B$, we have a direct sum decomposition
\begin{equation}\label{eq:decomposition}
A_k^\op \otimes_k B_k \simeq (A^\op \otimes_{k'}B)_k\oplus D
\end{equation}
of dg $k$-algebras.
\end{lemma}
\begin{proof}
One of the many equivalent characterizations of a finite {\em separable} field extension $k'/k$ asserts that $k'$ is projective as a $k'\otimes_kk'$-module; see \cite[Thm.~9.2.11]{Weibel}. Hence, the multiplication map $k'\otimes_k k' \to k'$ admits a section $s$ and so we obtain a central idempotent element $e:=s(1)$ of $k'\otimes_kk'$. Given dg $k'$-algebras $A$ and $B$ we have a canonical surjective map
\begin{equation}\label{eq:homo-mult}
A_k^\op \otimes_k B_k \twoheadrightarrow (A^\op \otimes_{k'}B)_k\,.
\end{equation} 
The dg $k$-algebra $A^\op_k \otimes_k B_k$ is naturally a dg $(k'\otimes_kk')$-algebra and so is endowed with a central idempotent $\tilde{e}:=1 \cdot e$. This allows us to define a section to \eqref{eq:homo-mult} by the rule $b_1\otimes b_2 \mapsto (b_1\otimes b_2) \cdot \tilde{e}$, and hence obtain the direct sum decomposition \eqref{eq:decomposition}.
\end{proof}
\begin{proposition}\label{prop:restriction}
The functor $(-)_k:\dgcat(k') \to \dgcat(k)$ preserves saturated dg categories.
\end{proposition}
\begin{proof}
Since by Lemma~\ref{lemma:key2} the functor $(-)_k$ preserves derived Morita equivalences, it suffices to consider saturated dg algebras; see \S\ref{sub:saturated}. Let $A$ be a saturated dg $k'$-algebra. On one hand, since by hypothesis the field extension $k'/k$ is finite, the sum $\sum_i \mathrm{dim}_k\textrm{H}^i(A_k)=[k':k]\cdot \sum_i \mathrm{dim}_{k'} \, \textrm{H}^i(A)$ remains finite. On the other hand, we have a well-defined triangulated functor
\begin{equation}\label{eq:functor-step}
\cD_c(A^\op \otimes_{k'} A) \simeq \cD_c((A^\op\otimes_{k'}A)_k) \too \cD(A^\op_k \otimes_k A_k)
\end{equation}
given by restriction along the canonical map $A^\op_k \otimes_k A_k \twoheadrightarrow (A^\op\otimes_{k'}A)_k$. By Lemma~\ref{lemma:key} (applied to $B=A$), $(A^\op \otimes_k A)_k$ is a direct factor of $A^\op\otimes_kA_k$ and hence a compact object. As a consequence, the above functor \eqref{eq:functor-step} take values in the subcategory $\cD_c(A^\op_k \otimes_k A_k)$ of compact objects. Moreover, it sends the $A\text{-}A$-bimodule $A$ to the $A_k\text{-}A_k$-bimodule $A_k$. Since by hypothesis $A$ belongs to $\cD_c(A^\op\otimes_{k'}A)$ we then conclude that $A_k$ belongs to $\cD_c(A^\op_k \otimes_kA_k)$. This shows that $A_k$ is a saturated dg $k$-algebra and so the proof is finished. 
\end{proof}
Proposition~\ref{prop:restriction} (and Proposition \ref{prop:induced}) implies that the above adjunction \eqref{eq:adju} restricts to saturated dg categories. By combining this fact with the description of the category of noncommutative Chow motives given in the proof of Proposition~\ref{prop:induced}, one obtains then the adjunction 
\begin{equation}\label{eq:adjun}
\xymatrix{
\NChow(k')_F \ar@<1ex>[d]^-{(-)_k} \\
\NChow(k)_F \ar@<1ex>[u]^{-\otimes_kk'}\,.
}
\end{equation}
Now, in order to obtain the adjunction on the right hand-side of \eqref{eq:adjunctions} it suffices then to show that both functors of \eqref{eq:adjun} descend to the category of noncommutative numerical motives. The case of the functor $-\otimes_kk'$ is contained in Theorem~\ref{thm:basechange1}. In what concerns $(-)_k$ we have the following result:
\begin{proposition}
We have a well-defined functor $(-)_k: \NNum(k')_F \to \NNum(k)_F$.
\end{proposition}
\begin{proof}
By combining Lemma~\ref{lemma:key2} with Proposition~\ref{prop:restriction} we obtain a well-defined functor $(-)_k:\sHmo_0(k') \to \sHmo_0(k)$. Hence, as explained in the proof of Theorem~\ref{thm:basechange1}, it suffices to show that it preserves the $\otimes$-ideal $\mathrm{Ker}(\chi)$, \ie given any two saturated dg $k'$-algebras $A$ and $B$ one needs to prove that the homomorphism
$$ \Hom_{\sHmo_0(k')}(A,B)=K_0(A^\op\otimes_{k'} B) \to K_0(A^\op_k\otimes_kB_k)= \Hom_{\sHmo_0(k)}(A_k, B_k)$$
preserves the kernel of the bilinear form $\chi$. Since $K_0(A^\op \otimes_{k'}B)$ is generated by the elements of shape $[X]$, with $X \in \cD_c(A^\op \otimes_{k'}B)$, and similarly $K_0(A_k^\op \otimes_k B_k)$ is generated by the elements $[Y]$, with $Y \in \cD_c(A_k^\op \otimes_kB_k)$, it suffices to prove the following implication:
\begin{equation}\label{eq:implication1}
\textrm{
$[X] \in \mathrm{Ker}(\chi) \Rightarrow \chi([Y],[X_k])=0 \,\, \textrm{for all} \,\,Y \in \cD_c(A_k^\op\otimes_kB_k)$.
}
\end{equation}
We have a natural adjunction
\begin{equation}\label{eq:adj-new}
\xymatrix{
\cD_c(A^\op\otimes_{k'}B) \ar@<1ex>[d]^-{(-)_k} \\
\cD_c(A_k^\op \otimes_k B_k)   \ar@<1ex>[u]^-{L}\,,
}
\end{equation}
where $L$ denotes the extension along the canonical map $A^\op_k \otimes_k B_k \twoheadrightarrow (A^\op \otimes_{k'}B)_k$ followed by the natural equivalence $\cD_c((A^\op \otimes_{k'}B)_k) \simeq \cD_c(A^\op\otimes_{k'}B)$. An argument similar to the one used in the proof of Lemma~\ref{lem:compact} shows us that $\chi([L(Y)],[X])=[k':k]\cdot \chi([Y],[X_k])$. Since $[k':k]\neq 0$, one concludes then that the above implication \eqref{eq:implication1} holds and so the proof is finished.
\end{proof}
\subsection{Compatibility}\label{sub:compatibility}
Given a field extension $k'/k$, recall from \cite[\S4.2.3]{Andre} the construction of the $F$-linear $\otimes$-functors
\begin{eqnarray}
\Chow(k)_F \too \Chow(k')_F && Z \mapsto Z_{k'} \label{eq:base-change-1}\\
\Num(k)_F \too \Num(k')_F && Z \mapsto Z_{k'} \label{eq:base-change-2}
\end{eqnarray}
between the categories of Chow and numerical motives. The compatibility of these base-change functors with those of Theorem~\ref{thm:basechange1} is the following:
\begin{theorem}\label{thm:compatibility}
When the field extension $k'/k$ is finite we have the following commutative diagrams
$$
\xymatrix{
\Chow(k')_F \ar[rr]^-{\Phi:=R\circ \pi} && \NChow(k')_F & \Num(k')_F \ar[rr]^-{\Phi_\cN:=R_\cN \circ \pi} && \NNum(k')_F \\
\Chow(k)_F \ar[rr]_-{\Phi:=R\circ \pi} \ar[u]^-{\eqref{eq:base-change-1}} && \NChow(k)_F \ar[u]_-{-\otimes_k k'} & \Num(k)_F  \ar[u]^-{\eqref{eq:base-change-2}} \ar[rr]_-{\Phi_\cN:= R_\cN \circ \pi} & & \NNum(k)_F \ar[u]_-{-\otimes_kk'}\,,
}
$$
where $\pi$, $R$, and $R_\cN$ are as in \eqref{eq:diagram-notation}. Assuming the standard sign conjecture $C^+$ (see \cite[\S5.1.3]{Andre}) as well as its noncommutative analogue $C_{NC}$, that $k$ is of characteristic zero, and that $k$ is a field extension of $F$ or that $F$ is a field extension of $k'$, we have moreover the following commutative diagram
\begin{equation}\label{eq:diag-com-dagger}
\xymatrix{\Num^\dagger(k')_F \ar[r]^-{\Phi_\cN} & \NNum^\dagger(k')_F \\
\Num^\dagger(k)_F \ar[u]^-{\eqref{eq:base-change-2}} \ar[r]_-{\Phi_\cN} & \NNum^\dagger(k)_F \ar[u]_-{-\otimes_k k'}\,.
}
\end{equation}
\end{theorem}
\begin{proof}
Let us start with the case of (noncommutative) Chow motives. Since the above functor \eqref{eq:base-change-1} maps the Tate motive $\bbQ(1)$ to itself, it descends to the orbit categories. Consequently, it suffices to show that the square
\begin{equation}\label{eq:base-change-final}
\xymatrix{
\Chow(k')_F/_{\!\!-\otimes \bbQ(1)} \ar[r]^-{R} & \NChow(k')_F \\
\Chow(k)_F/_{\!\!-\otimes\bbQ(1)} \ar[u]^-{\eqref{eq:base-change-1}} \ar[r]_-R & \NChow(k)_F \ar[u]_-{-\otimes_kk'}
}
\end{equation}
commutes. Recall from \cite[\S8]{CvsNC} the explicit construction of the fully faithful $F$-linear $\otimes$-functor $R$. Since $\cD_\perf^\dg(Z) \otimes_k k' \simeq \cD_\perf^\dg(Z_{k'})$ for any smooth projective $k$-scheme $Z$ (see \cite[Prop.~6.2]{Regularity}), one observes from {\em loc. cit.} that the commutativity of \eqref{eq:base-change-final} follows from Lemma~\ref{lem:aux-compatibility} below.

In what concerns (noncommutative) numerical motives, consider the diagram:
\begin{equation}\label{eq:diagram-big}
\xymatrix@C=.7em@R=3em{
\Chow(k')_F \ar[rrr]^-\Phi \ar[d]  &&& \NChow(k')_F \ar[d]  \\
\Num(k')_F  \ar[rrr]^-{\Phi_\cN} &&& \NNum(k')_F \\
 \Num(k)_F\ar[rrr]_-{\Phi_{\cN}} \ar[u]^-{\eqref{eq:base-change-2}}  &&&\NNum(k)_F \ar[u]_-{-\otimes_kk'} \\
\Chow(k)_F \ar[u] \ar[rrr]_-\Phi  \ar@/^5pc/[uuu]^-{\eqref{eq:base-change-1}}&&&  \NChow(k)_F \ar[u]  \ar@/_5pc/[uuu]_-{-\otimes_kk'}\,.
}
\end{equation}
All the six squares, except the middle one, are commutative. The four smallest ones are commutative by construction and the commutativity of the largest one was proved above. Recall that by construction the functor $\Chow(k)_F \to \Num(k)_F$ is full and that its image is {\em dense}, \ie every object of $\Num(k)_F$ is a direct factor of an object in the image. As a consequence, one concludes that the middle square in \eqref{eq:diagram-big} is also commutative.

Let us show that \eqref{eq:diag-com-dagger} is commutative. Recall from \cite[\S11]{Galois} the construction of the symmetric monoidal categories $\Num^\dagger(k)_F$ and $\Num^\dagger(k')_F$. Similarly to $\NNum^\dagger(k)_F$ and $\NNum^\dagger(k')_F$ (see the proof of Proposition \ref{prop:HP}), they are obtained by applying the general \cite[Prop.~9.7]{Galois} to the $F$-linear $\otimes$-functors
\begin{eqnarray}
& \Chow(k)_F \stackrel{\Phi}{\to} \NChow(k)_F \stackrel{\overline{HP_\ast}}{\to} \mathrm{sVect}(F) & \Chow(k)_F \to \Num(k)_F \label{eq:func-1} \\
& \Chow(k')_F \stackrel{\Phi}{\to} \NChow(k')_F \stackrel{\overline{HP_\ast}}{\to} \mathrm{sVect}(F) & \Chow(k')_F \to \Num(k')_F \label{eq:func-2}\,.
\end{eqnarray}
Recall from the proof of \cite[Thm.~1.3]{Galois} that the left-hand-sides of \eqref{eq:func-1}-\eqref{eq:func-2} identify with the $2$-perioditization of de Rham cohomology 
$$Z \mapsto (\underset{n\,\,\mathrm{even}}{\bigoplus} H^n_{dR}(Z), \underset{n \,\,\mathrm{odd}}{\bigoplus} H^n_{dR}(Z))\,.$$
The proof of our claim follows now from the combination of Theorem \ref{thm:basechange1} with the commutativity of the other two squares of Theorem~\ref{thm:compatibility}. 
\end{proof}
\begin{lemma}\label{lem:aux-compatibility}
Let $Z$ and $Z'$ be two smooth projective $k$-schemes. When the field extension $k'/k$ is finite the following square commutes
\begin{equation}\label{eq:square-last}
\xymatrix{
K_0(Z_{k'} \times Z'_{k'})_F \ar[rrr]^-{\mathrm{ch}(-)\cdot \pi^\ast_{Z'_{k'}}(\mathrm{td}(Z'_{k'}))} &&& \cZ_{\mathrm{rat}}^\ast(Z_{k'} \times Z'_{k'})_F \\
K_0(Z\times Z')_F \ar[rrr]_-{\mathrm{ch}(-)\cdot \pi^\ast_{Z'}(\mathrm{td}(Z'))} \ar[u] &&& \cZ_{\mathrm{rat}}^\ast(Z\times Z')_F \ar[u] \,,
}
\end{equation}
where $\mathrm{ch}$ denotes the Chern character and $\mathrm{td}$ the Todd class.
\end{lemma}
\begin{proof}
Recall first that we have a well-defined morphism $\iota\times\iota': Z_{k'}\times Z'_{k'} \to Z\times Z'$ of smooth projective $k$-schemes. Since the functors $K_0(-)_F$ and $\cZ_{\mathrm{rat}}^\ast(-)_F$ are independent of the base field the vertical homomorphisms of \eqref{eq:square-last} are given by
\begin{eqnarray*}
K_0(Z\times Z')_F \stackrel{(\iota\times\iota')^\ast}{\too} K_0(Z_{k'}\times Z'_{k'})_F &&
\cZ_{\mathrm{rat}}^\ast(Z\times Z')_F \stackrel{(\iota\times\iota')^\ast}{\too} \cZ_{\mathrm{rat}}^\ast(Z_{k'}\times Z'_{k'})_F \,.
\end{eqnarray*}
Therefore, if one ignores the Todd class the above square \eqref{eq:square-last} commutes since $\mathrm{ch}(-)$ is a natural transformation of functors between $K_0(-)_F$ and $\cZ_{\mathrm{rat}}^\ast(-)_F$. In what concerns the Todd class, recall from \cite[Example~18.3.9]{Fulton} that since by hypothesis the field extension $k'/k$ is finite we have $(\iota')^\ast(\mathrm{td}(Z'))=\mathrm{td}(Z'_{k'})$. Consequently, the following equality holds
$$
(\iota\times \iota')^\ast (\pi^\ast_{Z'}(\mathrm{td}(Z'))) =  \pi^\ast_{Z_{k'}}((\iota')^\ast(\mathrm{td}(Z')))=  \pi^\ast_{Z_{k'}}(\mathrm{td}(Z'_{k'}))\,.
$$
By combining it with the above arguments we conclude that \eqref{eq:square-last} commutes. This achieves the proof.
\end{proof}
\subsection{Motivic decompositions}\label{sub:decompositions}
Let ${\bf L} \in \Chow(k)_F$ be the Lefschetz motive (= the $\otimes$-inverse of the Tate motive $\bbQ(1)$) and ${\bf 1}:=\cU(k)$ the $\otimes$-unit of $\NChow(k)_F$. Consider the following natural notions:
\begin{itemize}
\item[(i)] A Chow motive in $\Chow(k)_F$ is called of {\em Lefschetz type} if it is isomorphic to ${\bf L}^{\otimes l_1} \oplus \cdots \oplus {\bf L}^{\otimes l_m}$ for some choice of non-negative integers $l_1, \ldots, l_m$.
\item[(ii)] A noncommutative Chow motive in $\NChow(k)_F$ is called of {\em unit type} if it is isomorphic to $\oplus_{i=1}^m {\bf 1}$ for a certain non-negative integer $m$.
\end{itemize}
Intuitively speaking, (i)-(ii) are the simplest kinds of (noncommutative) Chow motives. As explained in \cite[\S2]{MT}, the Chow motives of projective spaces, quadrics, rational surfaces, complex Fano threefolds with vanishing odd cohomology and certain classes of homogeneous spaces and moduli spaces, are of Lefschetz type. On the other hand, as proved in \cite[\S4]{MT}, whenever the derived category $\cD_\perf(Z)$ of perfect complexes of $\cO_Z$-complexes of a smooth projective $k$-scheme $Z$ admits a full exceptional collection (see \cite[\S1.4]{Huy}) the associated noncommutative Chow motive $\cU(\cD_\perf^\dg(Z))$ is of unit type. An important (and difficult) question in the world of Chow motives is the following:

\vspace{0.1cm}

{\it Question: Given a Chow motive $M \in \Chow(k)_F$, is there a field extension $k'/k$ under which $M_{k'}$ becomes of Lefschetz type ? }

\vspace{0.1cm}

Our work above provides new tools to address this question. As proved in \cite[Thms.~1.3 and 1.7]{MT}, a Chow motive $M \in \Chow(k)_F$ is of Lefschetz type if and only if the associated noncommutative Chow motive $\Phi(M) \in \NChow(k)_F$ is of unit type. By combining this equivalence with the above Theorem~\ref{thm:compatibility} we obtain the following result:
\begin{corollary}\label{cor:application-1}
Let $k'/k$ be an arbitrary field extension and $M \in \Chow(k)_F$ a Chow motive. Then, $M_{k'} \in \Chow(k')_F$ is of Lefschetz type if and only if the noncommutative Chow motive $\Phi(M) \otimes_k k'$ is of unit type.
\end{corollary}
\begin{proof}
If the field extension $k'/k$ is finite the proof is clear. On the other hand if $k'/k$ is not finite let us denote by $\{k_i\}_{i\in I}$ the filtered diagram of intermediate field extensions $k'/k_i/k$ that are finite over $k$. An argument similar to the one used in the proof of Proposition~\ref{prop:colim} furnish us equivalences of categories
\begin{eqnarray*}
\mathrm{colim}_i\Chow(k_i)_F \simeq \Chow(k')_F && \mathrm{colim}_i\NChow(k_i)_F \simeq \NChow(k')_F\,.
\end{eqnarray*}
Consequently, the isomorphisms $M_{k'}\simeq {\bf L}^{\otimes l_1} \oplus \cdots \oplus {\bf L}^{\otimes l_m}$ and $\Phi(M) \otimes_k k' \simeq \oplus_{i=1}^m {\bf 1}$ always occur at a field extension $k_i/k$ which is finite over $k$. This achieves the proof.
\end{proof}
Intuitively speaking, Corollary~\ref{cor:application-1} show us that in order to test if the Chow motive $M_{k'}$ is of Lefschetz type we can (alternatively) use the noncommutative world. Note that in the particular case where $M=M(Z)$, with $Z$ a smooth projective $k$-scheme, the equivalence of Corollary~\ref{cor:application-1} reduces to
\begin{equation}\label{eq:equivalence}
 M(Z_{k'})\,\,\mathrm{Lefschetz}\,\,\mathrm{type} \Leftrightarrow \cU(\cD_\perf^\dg(Z) \otimes_kk')\,\,\mathrm{unit}\,\,\mathrm{type}\,.
 \end{equation}
Hence, instead of testing $M(Z_{k'})$ one can test instead the noncommutative Chow motive $\cU(\cD_\perf^\dg(Z)\otimes_kk')$. For example, when $Z$ is a Severi-Brauer variety of dimension $n$ we have a semi-orthogonal decomposition\footnote{A generalization of the notion of full exceptional collection; see \cite[\S1.4]{Huy}.}
$$ \cD_\perf^\dg(Z) = \langle \cD_\perf^\dg(k), \cD_\perf^\dg(A),\ldots, \cD_\perf^\dg(A^{\otimes n}) \rangle\,,$$
where $A$ is the (noncommutative) Azumaya algebra associated to $Z$; see \cite{Marcello} for instance. As proved in \cite[\S5]{MT}, semi-orthogonal decompositions become direct sums in $\NChow(k)_F$. Moreover, $A\otimes_k\overline{k} \simeq \overline{k}$. Hence, one concludes that $\cU(\cD_\perf^\dg(Z)\otimes_k\overline{k})\simeq \oplus_{i=1}^{n+1} {\bf 1}$ and consequently, using the above equivalence \eqref{eq:equivalence}, that $M(Z_{\overline{k}})$ is of Lefschtez type. Since $Z$ is a Severi-Brauer variety we know a priori that $Z_{\overline{k}}\simeq \bbP_{\overline{k}}^{n+1}$ and hence that $M(Z_{\overline{k}})$ is of Lefschetz type. However, we believe this strategy can be generalized to other examples because the existence of semi-orthogonal decompositions is quite frequent; see \cite{Bondal-ICM,Huy,Rouquier} for instance.

Now, recall from the proof of Theorem~\ref{thm:basechange1} the notion of additive invariant. Thanks to the work of Blumberg-Mandell, Keller, Schlichting, Thomason-Trobaugh, Waldhausen, Weibel, and others (see \cite{BM,Exact,Exact2,Negative,Fundamental,MacLane-spectral,TT,Wald,Weibel-KH}), algebraic $K$-theory, Hochschild homology, cyclic homology, periodic cyclic homology, negative cyclic homology, topological Hochschild homology, and topological cyclic homology, are all examples of additive invariants. As proved in {\em loc. cit.}, when applied to $\cD_\perf^\dg(Z)$ these invariants reduce to the (classical) invariants of $Z$. Given an additive invariant $E$ let us then write $E(Z)$ instead of $E(\cD_\perf^\dg(Z))$.
\begin{corollary}\label{cor:final}
Let $Z$ be a smooth projective $k$-scheme such that $M(Z_{k'})$ is of Lefschetz type. Then, $E(Z_{k'})\simeq \oplus^m_{i=1}E(k')$ for every additive invariant $E$.
\end{corollary}
\begin{proof}
This follows from the combination of Corollary~\ref{cor:application-1} with the fact that every additive invariant $E$ descends to $\Hmo_0(k')_F$ and hence to $\NChow(k')_F$; see Theorem~\ref{thm:universal}.
\end{proof}

Corollary~\ref{cor:final} can be used in order to prove negative results. For instance, whenever one finds an additive invariant $E$ such that $E(Z_{k'})\neq \oplus_{i=1}^mE(k')$ one concludes that $M(Z_{k'})$ cannot be of Lefschetz type. Due to the numerous additive invariants in the literature we believe that Corollary~\ref{cor:final} is a pratical new tool to establish such negative results.
\subsection{Schur/Kimura finiteness}\label{sub:Schur}
Let $\cC$ be a $\bbQ$-linear, idempotent complete, symmetric monoidal category. Given a partition $\lambda$ of a natural number $n$, we can consider the associated Schur functor ${\bf S}_\lambda:\cC\to \cC$ which sends an object $c$ to the direct summand of $c^{\otimes n}$ determined by $\lambda$; see Deligne's foundational work \cite{Deligne}. We say that $c$ is {\em Schur-finite} if it is annihilated by some Schur functor. When $\lambda=(n)$, resp. $\lambda=(1,\ldots, 1)$, the associated Schur functor $\mathrm{Sym}^n:={\bf S}_{(n)}$, resp. $\mathrm{Alt}^n:={\bf S}_{(1,\ldots, 1)}$, should be considered as the analogue of the usual $n^{\mathrm{th}}$-symmetric, resp. $n^{\mathrm{th}}$-wedge, product of $\bbQ$-vector spaces. We say that $c$ is {\em evenly}, resp. {\em oddly}, {\em finite dimensional} if $\mathrm{Alt}^n(c)$, resp. $\mathrm{Sym}^n(c)$, vanishes for some $n$. Finally, $c$ is {\em Kimura-finite} if it admits a direct sum decomposition $c \simeq c_+ \oplus c_-$ into an evenly $c_+$ and an oddly $c_-$ finite dimensional object. Note that Kimura-finiteness implies Schur-finiteness. 

The categories $\Chow(k)_F$ and $\NChow(k)_F$ are by construction $\bbQ$-linear, idempotent complete, and symmetric monoidal. Hence, the above notions apply. In the commutative world these finiteness notions were extensively studied by Andr{\'e}-Kahn, Guletskii, Guletskii-Pedrini, Kimura, Mazza, and others; see~\cite{Andre,AK,Guletskii,GP,Kimura,Mazza}. For instance, Guletskii and Pedrini proved that given a smooth projective surface $S$ (over a field of characteristic zero) with $p_g(S)=0$, the Chow motive $M(S)$ is Kimura-finite if and only if Bloch's conjecture on the Albanese kernel for $S$ holds. Similarly to \S\ref{sub:decompositions}, we have the following important question:

\vspace{0.1cm}

{\it Question: Given a Chow motive $M \in \Chow(k)_F$, is there a field extension $k'/k$ under which $M_{k'}$ is Schur/Kimura-finite ?}

\vspace{0.1cm}

Our above work also provides new tools to attack this question. As proved in \cite[Thm.~2.1]{CvsNC}, a Chow motive $M \in \Chow(k)_F$ is Schur-finite if and only if the associated noncommutative Chow motive $\Phi(N) \in \NChow(k)_F$ is Schur-finite. Moreover, if $M$ is Kimura-finite then $\Phi(M)$ is also Kimura-finite. Thanks to Theorem~\ref{thm:compatibility} we then obtain the following result:
\begin{corollary}\label{cor:application-2}
Let $k'/k$ be an arbitrary field extension and $M\in \Chow(k)_F$ a Chow motive. Then, the following holds:
\begin{itemize}
\item[(i)] The Chow motive $M_{k'}\in \Chow(k')_F$ is Schur-finite if and only if the noncommutative Chow motive $\Phi(M)\otimes_kk'$ is Schur-finite.
\item[(ii)] If $M_{k'}$ is Kimura-finite, then $\Phi(M)\otimes_kk'$ is also Kimura-finite.
\end{itemize}
\end{corollary}
\begin{proof}
The proof is similar to the one of Corollary~\ref{cor:application-1}.
\end{proof}
We believe that Corollary~\ref{cor:application-2}, combined with a strategy similar to the one of \S\ref{sub:decompositions}, can be used in the search of new Schur-finite Chow motives.
\begin{corollary}\label{cor:very-last}
Let $Z$ be a smooth projective $k$-scheme such that $M(Z_{k'})$ is Schur/Kimura-finite. Then, $E(Z_{k'})$ is also Schur/Kimura-finite for every symmetric monoidal additive invariant $E$.
\end{corollary} 
\begin{proof}
This is simply the combination of Corollary~\ref{cor:application-2} with the following two facts: (i) every symmetric monoidal additive invariant $E$ descends to a symmetric monoidal additive $\bbQ$-linear functor defined on $\NChow(k)_F$; (ii) every symmetric monoidal additive $\bbQ$-linear functor preserves Schur/Kimura-finiteness. 
\end{proof}
Examples of symmetric monoidal additive invariants include, among others, topological Hochschild homology, topological cyclic homology, and Kassel's mixed complex construction\footnote{From which all variants of cyclic homology can be recovered.} (see \cite[\S5.3]{ICM}). These invariants take values respectively in the category of rational spectra, rational pro-spectra, and in the derived category of mixed complexes.

\begin{remark}
It was conjectured by Kimura \cite{Kimura} that every Chow motive is Kimura-finite (and hence Schur-finite). Using Corollary \ref{cor:application-2}, this implies that every additive invariant take values in Kimura finite objects.
\end{remark}
\section{Proof of Theorem~\ref{thm:main4}}\label{sec:proof4}
We need to prove that $P$ is surjective, that the composition $P\circ I$ is trivial, that $\mathrm{Ker}(P)\subseteq \mathrm{Im}(I)$, and that $I$ is injective.
\subsection*{Surjectivity of $P$}
This follows from items (ii) and (iii) of Theorem \ref{thm:main3}.
\subsection*{Triviality of the composition $P\circ I$}
Let $\varphi$ be an element of the noncommutative motivic Galois group $\mathrm{Gal}(\NNum^\dagger(\overline{k})_F)$, \ie a $\otimes$-automorphism of the fiber functor $\overline{HP_\ast}:\NNum^\dagger(\overline{k})_F \to \mathrm{Vect}(F)$. One needs to show that for every noncommutative Artin motive $L \in \NAM(k)_F$ the equality $\varphi_{L \otimes_k\overline{k}}=\id_{\overline{HP_\ast}(L\otimes_k\overline{k})}$ holds. Thanks to the description \eqref{eq:description} of the category $\NAM(k)_F$ (and Theorems~\ref{thm:main2}-\ref{thm:main3}) it suffices to consider the noncommutative Artin motives of the form $\cU(\cD_\perf^\dg(Z))$, with $Z$ a finite {\'e}tale $k$-scheme. Recall from \cite[\S1.3.1]{Andre} that $Z=\mathrm{spec}(A)$, with $A$ a finite dimensional {\'e}tale $k$-algebra. As a consequence, the dg category $\cD_\perf^\dg(Z)$ is naturally derived Morita equivalence to $\cD_\perf^\dg(A)$ and hence to $A$. The noncommutative Artin motive $\cU(\cD_\perf^\dg(Z))$ identifies then with $A$ (considered as a dg $k$-algebra concentrated in degree zero). Since $A$ is {\'e}tale, we have an isomorphism $A \otimes_k\overline{k}\simeq \overline{k}^{\mathrm{dim}(A)}$ of $k$-algebras, where $\mathrm{dim}(A)$ denote the dimension of $A$. Hence, the image of $A$ under the composed functor $\AM(k)_F \subset \NNum^\dagger(k)_F \stackrel{-\otimes_k\overline{k}}{\too} \NNum^\dagger(\overline{k})_F$ identifies with the direct sum of $\mathrm{dim}(A)$ copies of the tensor unit $\overline{k}$. This implies that $ \varphi_{A\otimes_k \overline{k}} = \id_{\overline{HP_\ast}(A\otimes_k \overline{k})}$ and so we conclude that the composition $P\circ I$ is trivial.
\subsection*{Injectivity of $I$}
We start with some general technical results. A subcategory $\cD_{\mathrm{dens}}$ of an idempotent complete category $\cD$ will be called {\em dense} if the pseudo-abelian envelope $(\cD_{\mathrm{dens}})^\natural$ agrees with $\cD$.
\begin{proposition}\label{prop:monomorphism}
Given a field $K$, let $\cC$ and $\cD$ be two $K$-linear abelian idempotent complete rigid symmetric monoidal categories, $H: \cC \to \cD$ an exact $K$-linear $\otimes$-functor, and $\omega: \cD \to \mathrm{Vect}(K)$ a fiber functor. Assume that $\cD$ is semi-simple and that for every object $Y$ of a (fixed) dense subcategory $\cD_{\mathrm{dens}}$ of $\cD$ there exists an object $X \in \cC$ and a surjective morphism $f: H(X) \twoheadrightarrow Y$. Then, the induced group homomorphism
\begin{eqnarray}\label{eq:induced-Galois}
\mathrm{Gal}(\cD):=\underline{\mathrm{Aut}}^\otimes(\omega) \too \underline{\mathrm{Aut}}^\otimes(\omega\circ H) =:\mathrm{Gal}(\cC) && \varphi \mapsto \varphi \circ H
\end{eqnarray}
is injective.
\end{proposition} 
\begin{proof}Let $\varphi \in \mathrm{Gal}(\cD)$ be an element in the kernel of \eqref{eq:induced-Galois}, \ie $\varphi_{H(X)}=\id_{\omega(H(X))}$ for any $X \in \cC$. We need to show that $\varphi_Y=\id_{\omega(Y)}$ for every object $Y$ of $\cD$. It suffices to show this claim for the objects of $\cD_{\mathrm{dens}}\subset \cD$ since every object in $\cD$ is a direct factor of an object in $\cD_{\mathrm{dens}}$. Given an object $Y \in \cD_{\mathrm{dens}}$, there exists by hypothesis an object $X \in \cC$ and a surjective morphism $f: H(X) \twoheadrightarrow Y$. Since $\cD$ is abelian semi-simple, $f$ admits by Lemma~\ref{lem:section} a section $s:Y \to H(X)$ and so one can consider the following commutative diagram
$$
\xymatrix{
\omega(Y) \ar[d]_-{\varphi_Y} \ar[r]^-{\omega(s)} & \omega(H(X)) \ar@{=}[d] \ar[r]^-{\omega(f)} & \omega(Y) \ar[d]^-{\varphi_Y} \\
\omega(Y)  \ar[r]_-{\omega(s)}& \omega(H(X)) \ar[r]_-{\omega(f)} & \omega(Y) \,,
}
$$
with $\omega(f) \circ \omega(s)=\id_{\omega(Y)}$. As a consequence, $\varphi_Y=\id_{\omega(Y)}$ and so the proof is finished.
\end{proof}
\begin{lemma}\label{lem:section}
Any surjective morphism $f:X \twoheadrightarrow Y$ in an abelian semi-simple $K$-linear category $\cD$ admits a section.
\end{lemma}
\begin{proof}
Since by hypothesis $\cD$ is abelian semi-simple, every object admits a unique finite direct sum decomposition into simple objects. Hence, $f$ corresponds to a surjective morphism $S_1 \oplus \cdots  \oplus S_n \twoheadrightarrow S_1 \oplus \cdots \oplus S_m$ between direct sums of simple objects. Since $\Hom_\cD(S_i,S_j)=\delta_{ij}\cdot K$, we conclude then that $f$ admits a section.
\end{proof}
In order to prove that $I$ is injective we will make use of the general Proposition~\ref{prop:monomorphism} (applied to $K=F$ and $H=-\otimes_k\overline{k}:\NNum^\dagger(k)_F \to \NNum^\dagger(\overline{k})_F$). By construction, the saturated dg $\overline{k}$-algebras form a dense subcategory of $\NNum^\dagger(\overline{k})_F$; see \S\ref{sub:numerical}. Hence, it suffices to show that for every saturated dg $\overline{k}$-algebra $B$ there exists a saturated dg $k$-algebra $A$ and a surjective morphism $f: A\otimes_k \overline{k} \twoheadrightarrow B$ in $\NNum^\dagger(\overline{k})_F$. The proof of this fact is now divided into a finite and an infinite case. Let us assume first that the field extension $\overline{k}/k$ is finite. Note that since $\overline{k}$ is algebraically closed and by hypothesis $k$ is of characteristic zero, the field extension $\overline{k}/k$ is separable. 
\begin{proposition}\label{prop:monadic}
The adjunction of Theorem~\ref{thm:basechange2} (with $k'=\overline{k}$)
\begin{equation}\label{eq:adjunction2}
\xymatrix{
\NNum(\overline{k})_F \ar@<1ex>[d]^-{(-)_k}  \\
\NNum(k)_F \ar@<1ex>[u]^-{-\otimes_k\overline{k}}\,
}
\end{equation}
is monadic (see \cite[\S VI]{MacLane}). Moreover, we have a natural isomorphism of functors $-\otimes_k \overline{k}_k \simeq (-\otimes_k \overline{k})_k$, where $\overline{k}_k$ denotes the $k$-algebra $\overline{k}$.
\end{proposition}
\begin{proof}
Recall that since by hypothesis $k$, $\overline{k}$, and $F$ have the same characteristic the categories $\NNum(k)_F$ and $\NNum(\overline{k})_F$ are not only abelian but moreover semi-simple. As a consequence the additive functor $(-)_k$ is in fact exact, \ie it preserves all kernels and cokernels. We now show that it is moreover conservative, \ie that it reflects all isomorphisms. These two properties imply in particular that $(-)_k$ reflects all cokernels and so Beck's monadic conditions (see \cite[page 151]{MacLane}) are satisfied. Let $M \in \NNum(\overline{k})_F$ be a noncommutative numerical motive. Since $(-)_k$ is exact, in order to prove that it is moreover conservative, it suffices to show that $M\simeq 0$ whenever $M_k\simeq 0$. Let us assume first that $M$ is a saturated dg $\overline{k}$-algebra $B$. If by hypothesis the $F$-vector spaces
\begin{equation}\label{eq:Hom2}
\Hom_{\NNum(k)_F}(A_k,B_k) = K_0(A_k^\op \otimes_k B_k)_F/\mathrm{Ker}(\chi_{A_k,B_k}) 
\end{equation}
are trivial for every saturated dg $\overline{k}$-algebra $A$, one needs to show that the $F$-vector spaces 
\begin{equation}\label{eq:Hom1}
\Hom_{\NNum(\overline{k})_F}(A,B) = K_0(A^\op \otimes_{\overline{k}}B)_F/\mathrm{Ker}(\chi_{A,B})
\end{equation}
are also trivial. Lemma \ref{lemma:key} (with $k'=\overline{k}$) furnish us a direct sum decomposition $A_k^\op \otimes_k B_k \simeq (A^\op \otimes_{\overline{k}}B)_k \oplus D$. As a consequence, \eqref{eq:Hom2} identifies with
$$ \left(K_0(A^\op\otimes_{\overline{k}} B)_F \oplus K_0(D)_F \right)/\mathrm{Ker}(\chi_{A,B}+\chi_D)\,, $$
where $\chi_{A,B} + \chi_D$ denotes the sum of the two bilinear forms. Moreover, the homomorphism from \eqref{eq:Hom1} to \eqref{eq:Hom2} induced by $(-)_k$ corresponds to the inclusion. Hence, since by hypothesis the $F$-vector spaces \eqref{eq:Hom2} are trivial one concludes that \eqref{eq:Hom1} is also trivial. Let us now assume that $M$ is of the form $(B,e)$, with $B$ a saturated dg $\overline{k}$-algebra and $e$ an idempotent of the $F$-algebra $K_0(B^\op \otimes_{\overline{k}}B)_F/\mathrm{Ker}(\chi)$. If by hypothesis $(B,e)_k \simeq 0$, then $e_k=0$ in the $F$-algebra $K_0(B_k^\op \otimes_k B_k)_F/\mathrm{Ker}(\chi)$. As proved above the homomorphism
$$ K_0(B^\op \otimes_{\overline{k}}B)_F/\mathrm{Ker}(\chi) \too \K_0(B_k^\op \otimes_k B_k)_F/\mathrm{Ker}(\chi)$$
is injective and so one concludes that $e$, and hence $(B,e)$, is also trivial.

In what concerns the isomorphism of functors, note that we have a natural transformation $\Id \otimes\epsilon_{\overline{k}}: -\otimes_k \overline{k}_k \otimes_k \overline{k} \Rightarrow -\otimes_k \overline{k}$, where $\epsilon$ denotes the counit of the adjunction. Via the above adjunction \eqref{eq:adjunction2} it corresponds to a natural transformation $-\otimes_k \overline{k}_k \Rightarrow (-\otimes_k\overline{k})_k$. This latter one is an isomorphism since the functors $-\otimes_k \overline{k}$ and $(-)_k$ do not alter the set of objects and the same isomorphism holds for complexes of vector spaces.
\end{proof}

Returning to the proof of the injectivity of $I$, take for $A$ the saturated dg $\overline{k}$-algebra $B_k$ and for $f$ the counit $\epsilon_B: B_k \otimes_k \overline{k} \to B$ of the above adjunction \eqref{eq:adjunction2}. This adjunction is monadic and so $\epsilon_B$ is in fact surjective. Note that the categories $\NNum(\overline{k})_F$ and $\NNum^\dagger(\overline{k})_F$ are the same (except in what concerns the symmetry isomorphism constraints) and so $f$ is also surjective in $\NNum^\dagger(\overline{k})_F$.

Let us now assume that the field extension $\overline{k}/k$ is infinite. Consider the filtrant diagram $\{k_i\}_{i \in I}$ of intermediate field extensions $\overline{k}/k_i/k$ which are finite over $k$. Since $\overline{k}$ identifies with $\mathrm{colim}_{i\in I} k_i$ there exists by \cite[Thm.~1.1]{Toen} a finite field extension $k_0/k$ and a saturated dg $k_0$-algebra $B_0$ such that $B_0 \otimes_{k_0}\overline{k}\simeq B$. Moreover, since by hypothesis $k$ is of characteristic zero the field extensions $\overline{k}/k_0$ and $\overline{k}/k$ are separable and so one concludes that $k_0/k$ is also a separable field extension. Hence, using the same argument as above (with $\overline{k}$ replaced by $k_0$), one obtains a surjective morphism $f: (B_0)_k \otimes_k k_0 \twoheadrightarrow B_0$ in $\NNum^\dagger(k_0)_F$. Since the category $\NNum^\dagger(k_0)_F$ is abelian semi-simple, Lemma~\ref{lem:section} implies that $f$ admits a section. As a consequence, the image of $f$ under the base-change functor $-\otimes_{k_0}\overline{k}: \NNum^\dagger(k_0)_F \to \NNum^\dagger(\overline{k})_F$ admits also a section and hence is surjective. In sum, we have a surjective morphism $(B_0)_k \otimes_k \overline{k} \twoheadrightarrow B_0\otimes_{k_0}\overline{k} \simeq B$ in $\NNum^\dagger(\overline{k})_F$. This concludes the proof of the injectivity of $I$.

\subsection*{Inclusion $\mathrm{Ker}(P)\subseteq \mathrm{Im}(I)$}
Recall first that since by hypothesis $F$ is a field extension of $\overline{k}$ we have a well-defined commutative diagram
$$
\xymatrix{
\NNum^\dagger(k)_F \ar[dr]_-{\overline{HP_\ast}} \ar[r]^-{-\otimes_k\overline{k}} & \NNum^\dagger(\overline{k})_F \ar[d]^-{\overline{HP_\ast}} \\ 
& \mathrm{Vect}(F)\,.
}
$$
Let $\varphi$ be an element of the noncommutative motivic Galois group $\mathrm{Gal}(\NNum^\dagger(k)_F)$ such that $\varphi_L=\id_{\overline{HP_\ast}(L)}$ for every noncommutative Artin motive $L \in \NAM(k)_F$. One needs to construct an element $\overline{\varphi}$ of $\mathrm{Gal}(\NNum^\dagger(\overline{k})_F)$, \ie a $\otimes$-automorphism of the fiber functor $\overline{HP_\ast}: \NNum^\dagger(\overline{k})_F \to \mathrm{Vect}(F)$, such that $\overline{\varphi}_{M \otimes_k\overline{k}}=\varphi_M$ for every noncommutative numerical motive $M \in \NNum^\dagger(k)_F$. The proof is divided into a finite and an infinite case.  Let us assume first that the field extension $\overline{k}/k$ is finite. Recall from Proposition~\ref{prop:monadic} that we have a monadic adjunction
\begin{equation}\label{eq:adjunction3}
\xymatrix{
\NNum(\overline{k})_F \ar@<1ex>[d]^-{(-)_k}  \\
\NNum(k)_F \ar@<1ex>[u]^-{-\otimes_k\overline{k}}\,.
}
\end{equation}
Let us denote by $\eta: \Id \Rightarrow (-\otimes_k\overline{k})_k$ and $\epsilon: (-)_k \otimes_k \overline{k} \Rightarrow \Id$ the unit and counit of the adjunction \eqref{eq:adjunction3}. Since this adjunction is monadic every noncommutative numerical motive $N \in \NNum(\overline{k})_F$ admits the following canonical presentation
\begin{equation}\label{eq:cok}
\mathrm{cokernel}\left(\epsilon_{N_k\otimes_k \overline{k}} -(\epsilon_N)_k \otimes_k \overline{k}: (N_k \otimes_k \overline{k})_k \otimes_k \overline{k} \too N_k\otimes_k \overline{k} \right) \simeq N\,;
\end{equation}
see \cite[\S7]{MacLane}. The same holds in $\NNum^\dagger(\overline{k})_F$ since this category only differs from $\NNum(\overline{k})_F$ in what concerns the symmetry isomorphism constraints. 
\begin{lemma}\label{lem:auxiliar}
The following square commutes
\begin{equation}\label{eq:diag1}
\xymatrix{
\overline{HP_\ast}((N_k \otimes_k \overline{k})_k \otimes_k \overline{k}) \ar[d]^-\simeq_{\varphi_{(N_k \otimes_k \overline{k})_k}} \ar[rrrrr]^-{\overline{HP_\ast}(\epsilon_{N_k \otimes_k \overline{k}}) - \overline{HP_\ast}((\epsilon_N)_k \otimes_k \overline{k})} &&&&&\overline{HP_\ast}(N_k \otimes_k \overline{k}) \ar[d]_\simeq^{\varphi_{N_k}} \\
\overline{HP_\ast}((N_k \otimes_k \overline{k})_k \otimes_k \overline{k}) \ar[rrrrr]_-{\overline{HP_\ast}(\epsilon_{N_k \otimes_k \overline{k}}) - \overline{HP_\ast}((\epsilon_N)_k \otimes_k \overline{k})} &&&&& \overline{HP_\ast}(N_k \otimes_k \overline{k})\,.
}
\end{equation}
\end{lemma}
\begin{proof}
Note first that by naturally of $\varphi$ we have
$$ \varphi_{N_k} \circ \overline{HP_\ast}((\epsilon_N)_k \otimes_k \overline{k}) = \overline{HP_\ast}((\epsilon_N)_k \otimes_k \overline{k}) \circ \varphi_{(N_k \otimes_k \overline{k})_k}\,.$$
Hence, the commutativity of \eqref{eq:diag1} follows from the commutativity of the square
\begin{equation}\label{eq:diag3}
\xymatrix{
\overline{HP_\ast}((N_k \otimes_k \overline{k})_k \otimes_k \overline{k}) \ar[d]^-\simeq_{\varphi_{(N_k \otimes_k \overline{k})_k}} \ar[rrr]^-{\overline{HP_\ast}(\epsilon_{N_k \otimes_k \overline{k}})} &&&\overline{HP_\ast}(N_k \otimes_k \overline{k}) \ar[d]_-\simeq^-{\varphi_{N_k}} \\
\overline{HP_\ast}((N_k \otimes_k \overline{k})_k \otimes_k \overline{k}) \ar[rrr]_-{\overline{HP_\ast}(\epsilon_{N_k \otimes_k \overline{k}})} &&&\overline{HP_\ast}(N_k \otimes_k \overline{k})\,.
}
\end{equation}
Let us now prove that \eqref{eq:diag3} is commutative. The following natural isomorphisms hold:
\begin{equation*}
(N_k \otimes_k \overline{k})_k \otimes_k \overline{k} \simeq (N_k \otimes_k \overline{k}_k) \otimes_k \overline{k} \simeq (N_k \otimes_k\overline{k})\otimes_k(\overline{k}_k \otimes_k \overline{k})\,.
\end{equation*}
The first one follows from the isomorphism of functors $(-\otimes_k \overline{k})_k \simeq -\otimes_k \overline{k}_k$ of Proposition~\ref{prop:monadic} and the second one from the fact that the functor $-\otimes_k \overline{k}$ is symmetric monoidal. Under these isomorphisms, $\overline{HP_\ast}(\epsilon_{N_k \otimes_k \overline{k}})$ identifies with 
$$ \overline{HP_\ast}(N_k\otimes_k\overline{k}) \otimes \overline{HP_\ast}(\overline{k}_k \otimes_k \overline{k}) \stackrel{\id \otimes \overline{HP_\ast}(\epsilon_{\overline{k}})}\too \overline{HP_\ast}(N_k\otimes_k\overline{k}) \otimes \overline{HP_\ast}(\overline{k}) \simeq \overline{HP_\ast}(N_k \otimes_k \overline{k})$$
and $\varphi_{(N_k \otimes_k \overline{k})_k}$ with
$$ \overline{HP_\ast}(N_k \otimes_k \overline{k}) \otimes \overline{HP_\ast}(\overline{k}_k \otimes_k \overline{k}) \stackrel{\varphi_{N_k} \otimes \varphi_{\overline{k}_k}}{\too} \overline{HP_\ast}(N_k \otimes_k \overline{k}) \otimes \overline{HP_\ast}(\overline{k}_k \otimes_k \overline{k})\,.$$
Now, note that $\overline{k}_k$ is a finite {\'e}tale $k$-algebra and hence a noncommutative Artin motive. Since by hypothesis $\varphi_L=\id_{\overline{HP_\ast}(L)}$ for every noncommutative Artin motive we have $\varphi_{\overline{k}_k}=\id_{\overline{HP_\ast}(\overline{k}_k)}$. By combining all these facts we then conclude that the above diagram \eqref{eq:diag3} is commutative since it identifies with 
$$
\xymatrix{
\overline{HP_\ast}(N_k \otimes_k \overline{k}) \otimes \overline{HP_\ast}(\overline{k}_k \otimes_k \overline{k}) \ar[d]^-\simeq_-{\varphi_{N_k}\otimes \id} \ar[rrr]^-{\id \otimes \overline{HP_\ast}(\epsilon_{\overline{k}})} &&&  \overline{HP_\ast}(N_k \otimes_k \overline{k}) \ar[d]_-\simeq^-{\varphi_{N_k}} \\
\overline{HP_\ast}(N_k \otimes_k \overline{k}) \otimes \overline{HP_\ast}(\overline{k}_k \otimes_k \overline{k}) \ar[rrr]_-{\id \otimes \overline{HP_\ast}(\epsilon_{\overline{k}})} &&& \overline{HP_\ast}(N_k \otimes_k \overline{k})\,.
}
$$
\end{proof}
The fiber functor $\overline{HP_\ast}:\NNum^\dagger(\overline{k})_F \to \mathrm{Vect}(F)$ is exact and symmetric monoidal. Hence, by combining Lemma~\ref{lem:auxiliar} with the canonical presentation \eqref{eq:cok}, one obtains an automorphism $\overline{\varphi}_N:\overline{HP_\ast}(N) \stackrel{\sim}{\to} \overline{HP_\ast}(N)$ for every noncommutative numerical motive $N \in \NNum^\dagger(\overline{k})_F$. By construction these automorphisms are functorial on $N$ and hence give rise to a well-defined automorphism $\overline{\varphi}$ of the fiber functor $\overline{HP_\ast}$. 

Let us now show that this automorphism $\overline{\varphi}$ is tensorial. Given any two noncommutative numerical motives $N, O \in \NNum^\dagger(\overline{k})_F$ one needs to show that the following square commutes 
\begin{equation}\label{eq:compatibility}
\xymatrix{
\overline{HP_\ast}(N) \otimes \overline{HP_\ast}(O) \ar[r]^-\simeq \ar[d]^-\simeq_{\overline{\varphi}_{N}\otimes \overline{\varphi}_{O}} & \overline{HP_\ast}(N \otimes_{\overline{k}}O) \ar[d]_-\simeq^-{\overline{\varphi}_{N\otimes_{\overline{k}} O}} \\
\overline{HP_\ast}(N) \otimes \overline{HP_\ast}(O) \ar[r]_-\simeq & \overline{HP_\ast}(N\otimes_{\overline{k}}O)\,.
}
\end{equation}
In order to prove this consider the following diagram\,:
$$
\xymatrix@C=.7em@R=3em{
\overline{HP_\ast}(N_k \otimes_k \overline{k}) \otimes \overline{HP_\ast}(O_k \otimes_k \overline{k}) \ar[rrr]^-\simeq \ar[d]^-{\overline{HP_\ast}(\epsilon_{N_k}) \otimes \overline{HP_\ast}(\epsilon_{O_k})} \ar@/_5pc/[ddd]^-\simeq_-{\varphi_{N_k} \otimes\varphi_{O_k}} &&& \overline{HP_\ast}((N_k\otimes_kO_k)\otimes_k \overline{k}) \ar[d]_-{\overline{HP_\ast}(\epsilon_{N_k} \otimes \epsilon_{O_k})}  \ar@/^5pc/[ddd]_\simeq^-{\varphi_{N_k \otimes _k O_k}}\\
 \overline{HP_\ast}(N) \otimes\overline{HP_\ast}(O) \ar[d]^-\simeq_-{\overline{\varphi}_N \otimes \overline{\varphi}_O} \ar[rrr]^-\simeq &&& \overline{HP_\ast}(N\otimes_{\overline{k}}O) \ar[d]_\simeq^-{\overline{\varphi}_{N\otimes_{\overline{k}}O}} \\
  \overline{HP_\ast}(N) \otimes \overline{HP_\ast}(O) \ar[rrr]_-\simeq &&& \overline{HP_\ast}(N\otimes_{\overline{k}}O) \\
 \overline{HP_\ast}(N_k \otimes_k \overline{k}) \otimes \overline{HP_\ast}(O_k \otimes_k \overline{k}) \ar[u]_-{\overline{HP_\ast}(\epsilon_{N_k}) \otimes \overline{HP_\ast}(\epsilon_{O_k})} \ar[rrr]_-\simeq &&&  \overline{HP_\ast}((N_k\otimes_kO_k)\otimes_k \overline{k}) \ar[u]^-{\overline{HP_\ast}(\epsilon_{N_k} \otimes \epsilon_{O_k})} \,.
}
$$
All the six squares, except the middle one (\ie \eqref{eq:compatibility}), are commutative. This follows from the hypothesis that $\varphi$ is tensorial, from the naturalness of $\overline{\varphi}$, and from the fact that the fiber functor $\overline{HP_\ast}$ is symmetric monoidal. As a consequence, the two composed morphisms of \eqref{eq:compatibility} (from $\overline{HP_\ast}(N) \otimes \overline{HP_\ast}(O)$ to $\overline{HP_\ast}(N\otimes_{\overline{k}}O)$) agree when pre-composed with $\overline{HP_\ast}(\epsilon_{N_k}) \otimes \overline{HP_\ast}(\epsilon_{O_k})$. Both counit morphisms $\epsilon_{N_k}$ and $\epsilon_{O_k}$ are surjective and hence by Lemma~\ref{lem:section} admit sections. This implies that $\overline{HP_\ast}(\epsilon_{N_k})\otimes \overline{HP_\ast}(\epsilon_{O_k})$ is also surjective and so one concludes that the above square \eqref{eq:compatibility} commutes.

Let us now show that $\overline{\varphi}_{M \otimes_k \overline{k}}=\varphi_M$ for every noncommutative numerical motive $M \in \NNum^\dagger(k)_F$. By construction of $\overline{\varphi}_{M\otimes_k \overline{k}}$ it suffices to show that the following square
\begin{equation}\label{eq:diagram-new}
\xymatrix{
\overline{HP_\ast}((M\otimes_k \overline{k})_k \otimes_k \overline{k}) \ar[d]^-\simeq_-{\varphi_{(M \otimes_k \overline{k})_k}} \ar[rr]^-{\overline{HP_\ast}(\epsilon_{M \otimes_k \overline{k}})} && \overline{HP_\ast}(M \otimes_k \overline{k}) \ar[d]_-\simeq^-{\varphi_M} \\
\overline{HP_\ast}((M\otimes_k \overline{k})_k \otimes_k \overline{k}) \ar[rr]_-{\overline{HP_\ast}(\epsilon_{M \otimes_k \overline{k}})} && \overline{HP_\ast}(M \otimes_k \overline{k})
}
\end{equation} 
commutes. As in the proof of Lemma~\ref{lem:auxiliar} this is indeed the case since \eqref{eq:diagram-new} identifies with the following commutative square
$$
\xymatrix{
\overline{HP_\ast}(M \otimes_k \overline{k}) \otimes \overline{HP_\ast}(\overline{k}_k \otimes_k \overline{k}) \ar[rr]^-{\id \otimes \overline{HP_\ast}(\epsilon_{\overline{k}})} \ar[d]^-\simeq_-{\varphi_M \otimes \id} && \overline{HP_\ast}(M \otimes_k \overline{k}) \ar[d]_-\simeq^-{\varphi_M} \\
 \overline{HP_\ast}(M \otimes_k \overline{k}) \otimes \overline{HP_\ast}(\overline{k}_k \otimes_k \overline{k}) \ar[rr]_-{\id \otimes \overline{HP_\ast}(\epsilon_{\overline{k}})} &&  \overline{HP_\ast}(M \otimes_k \overline{k})\,.
}
$$
Let us now assume that the field extension $\overline{k}/k$ is infinite. Consider the filtered diagram $\{k_i\}_{i \in I}$ of intermediate field extensions $\overline{k}/k_i/k$ which are finite over $k$. By applying the usual strictification procedure one obtains a filtrant diagram of categories $\{\NNum^\dagger(k_i)_F\}_{i\in I}$. The base-change $\otimes$-functors $-\otimes_{k_i}\overline{k}: \NNum^\dagger(k_i)_F \to \NNum^\dagger(\overline{k})_F, i \in I$, give rise to a well-defined $\otimes$-functor
$$ \Lambda: \mathrm{colim}_{i \in I} \NNum^\dagger(k_i)_F \too \NNum^\dagger(\overline{k})_F\,.$$
\begin{proposition}\label{prop:colim}
The functor $\Lambda$ is an equivalence.
\end{proposition}
\begin{proof}
We start by proving that $\Lambda$ is fully-faithful. Let $M$ and $N$ be two objects of $\mathrm{colim}_{i \in I} \NNum^\dagger(k_i)_F$. There exists then an element $j \in I$ and objects $M_j, N_j \in \NNum^\dagger(k_j)_F$ such that $M_j \otimes_{k_j}\overline{k}=M$ and $N_j \otimes_{k_j}\overline{k}=N$. By construction of the category of noncommutative numerical motives it suffices to consider the case where $M_j$ and $N_j$ are saturated dg $k_j$-algebras. We have the following equalities
\begin{eqnarray*}
\Hom_{\mathrm{colim}_{i \in I}\NNum^\dagger(k_i)_F}(M,N) & = & \mathrm{colim}_{j \downarrow I} \Hom_{\NNum^\dagger(k_i)_F}(M_j \otimes_{k_j}k_i,N_j \otimes_{k_j}k_i)\\
& = & \mathrm{colim}_{j \downarrow I} \left(K_0((M_j^\op \otimes_{k_i} N_j)\otimes_{k_j}k_i)_F / \mathrm{Ker}(\chi)\right)\,,
\end{eqnarray*}
where $j \downarrow I$ denotes the category of objects under $j$. On the other hand we have
\begin{eqnarray*}
\Hom_{\NNum^\dagger(\overline{k})_F}(\Lambda(M),\Lambda(N)) & = & \Hom_{\NNum^\dagger(\overline{k})_F}(M_j \otimes_{k_j} \overline{k}, N_j \otimes_{k_j} \overline{k})\\
& = & K_0((M_j^\op \otimes_{k_j}N_j)\otimes_{k_j} \overline{k})_F/\mathrm{Ker}(\chi)\,.
\end{eqnarray*}
Proposition~\ref{prop:key} (applied to $k=k_j$, $k'=\overline{k}$, $C=M_j^\op \otimes_{k_j} N_j$, and $I=j \downarrow I$) furnish us an equivalence of categories
$$ \mathrm{colim}_{j \downarrow I}\cD_c((M_j^\op \otimes_{k_j}N_j)\otimes_{k_j}k_i) \stackrel{\sim}{\too} \cD_c((M_j^\op \otimes_{k_j} N_j)\otimes_{k_j}\overline{k})\,.$$
Since the Grothendieck group functor preserves filtered colimits the upper horizontal map in the following commutative diagram is an isomorphism
$$ 
\xymatrix@C=1em@R=2em{
\mathrm{colim}_{j \downarrow I}K_0((M_j^\op \otimes_{k_j}N_j)\otimes_{k_j}k_i)_F \ar[d] \ar[r]^-\simeq & K_0((M_j^\op \otimes_{k_j}N_j)\otimes_{k_j}\overline{k})_F \ar@{->>}[d] \\
\mathrm{colim}_{j \downarrow I}\left(K_0((M_j^\op \otimes_{k_j}N_j)\otimes_{k_j}k_i)_F/\mathrm{Ker}(\chi)\right)  \ar[r] & K_0((M_j^\op \otimes_{k_j}N_j)\otimes_{k_j}\overline{k})_F/\mathrm{Ker}(\chi) \,.
}
$$
The right vertical map is surjective and so we conclude that the lower horizontal map is also surjective. Since this latter map identifies with the one from  $\Hom_{\mathrm{colim}_{i \in I}\NNum^\dagger(k_i)_F}(M,N)$ to $\Hom_{\NNum^\dagger(\overline{k})_F}(\Lambda(M),\Lambda(N))$ induced by $\Lambda$ it remains to show that it is moreover injective. Let $X_i$ be an object of $\cD_c((M_j^\op \otimes_{k_j}N_j)\otimes_{k_j}k_i)$ such that $\chi([X_i \otimes_{k_i} \overline{k}],[Y])=0$ for every $Y \in \cD_c((M_j^\op \otimes_{k_j}N_j)\otimes_{k_j}\overline{k})$. Given any object $W \in \cD_c((M^\op_j\otimes_{k_j} N_j)\otimes_{k_j}k_i)$ we have by Proposition~\ref{prop:key} the equality $\chi([X_i],[W]) =\chi([X_i \otimes_{k_i}\overline{k}],[W\otimes_{k_i}\overline{k}])$. As a consequence $[X_i]$ becomes trivial in $\mathrm{colim}_{j \downarrow I}\left(K_0((M_j^\op \otimes_{k_j}N_j)\otimes_{k_j}k_i)_F/\mathrm{Ker}(\chi)\right)$ and so we conclude that the lower horizontal map is injective.

We now prove that $\Lambda$ is essentially surjective. Let $N$ be an object of $\NNum^\dagger(\overline{k})_F$. By definition $N$ is of the form $(B,e)$, with $B$ a saturated dg $\overline{k}$-algebra and $e$ an idempotent element of the $F$-algebra $K_0(B^\op \otimes_{\overline{k}}B)_F/\mathrm{Ker}(\chi)$. As explained above, there exists a finite field extension $k_0/k$ and a saturated dg $k_0$-algebra $B_0$ such that $B_0 \otimes_{k_0}\overline{k}\simeq B$. The functor $\Lambda$ is fully-faithful and so we have an isomorphism
$$ \mathrm{colim}_{j \downarrow I} \left(K_0((B_0^\op \otimes_{k_0}B_0)\otimes_{k_0}k_i)_F/\mathrm{Ker}(\chi) \right)\stackrel{\sim}{\too} K_0(B^\op \otimes_{\overline{k}}B)_F/\mathrm{Ker}(\chi)\,.$$
Now, since the maps in the filtered diagram respect the composition operation there exists a finite field extension $k_l/k_j$ and an idempotent element $e_l$ of the $F$-algebra $K_0((B_0^\op \otimes_{k_0}B)\otimes_{k_0}k_l)_F/\mathrm{Ker}(\chi)$ such that $e_l \otimes_{k_l}\overline{k} = e$. As a consequence the noncommutative numerical motive $(B_0\otimes_{k_0}k_l,e_l)$ is such that $(B_0\otimes_{k_0}k_l, e_l)\otimes_{k_l}\overline{k} \simeq (B,e)$. This shows that the functor $\Lambda$ is essentially surjective and so the proof is finished.
\end{proof}
Returning to the proof of the inclusion $\mathrm{Ker}(P)\subseteq \mathrm{Im}(I)$, recall that we have by hypothesis an element $\varphi \in \mathrm{Gal}(\NNum^\dagger(k)_F)$ such that $\varphi_L=\id_{\overline{HP_\ast}(L)}$ for every noncommutative Artin motive $L \in \NAM(k)_F$. As proved above (in the case of a finite field extension $\overline{k}/k$), the $\otimes$-automorphism $\varphi$ can be ``lifted'' to $\otimes$-automorphisms $\overline{\varphi^i} \in \mathrm{Gal}(\NNum^\dagger(k_i)_F), i \in I$, in the sense that $\overline{\varphi^i}_{M\otimes_k k_i}= \varphi_M$ for every noncommutative numerical motive $M \in \NNum^\dagger(k)_F$. Since the homomorphisms $\mathrm{Gal}(\NNum^\dagger(k_i)_F) \to \mathrm{Gal}(\NNum^\dagger(k)_F), i \in I$, are injective one concludes then that the following equalities hold:
\begin{eqnarray}\label{eq:equalities}   
\overline{\varphi^j} \circ -\otimes_{k_i}k_j = \overline{\varphi^i} & & k_i \subseteq k_j \,.
\end{eqnarray}
As a consequence the $\otimes$-automorphisms $\overline{\varphi^i}, i \in I$, assemble into a well-defined $\otimes$-automorphism $\overline{\varphi}$ of the fiber functor $\overline{HP_\ast}: \mathrm{colim}_{i \in I} \NNum^\dagger(k_i)_F \to \mathrm{Vect}(F)$. By Proposition~\ref{prop:colim}, $\mathrm{colim}_{i \in I} \NNum^\dagger(k_i)_F$ identifies with $\NNum^\dagger(\overline{k})_F$ and so we obtain a well-defined element $\overline{\varphi}$ of $\mathrm{Gal}(\NNum^\dagger(\overline{k})_F)$ such that $\overline{\varphi}_{M \otimes_k \overline{k}}=\varphi_M$ for every noncommutative numerical motive $M \in \NNum^\dagger(k)_F$. This concludes the proof of the inclusion $\mathrm{Ker}(P) \subseteq \mathrm{Im}(I)$.

\begin{remark}
When the field extension $\overline{k}/k$ is infinite one can use Proposition~\ref{prop:colim} to give an alternative proof of the injectivity of $I$: this follows from the fact that the noncommutative motivic Galois group $\mathrm{Gal}(\NNum^\dagger(\overline{k})_F)$ identifies with $\mathrm{lim}_{i \in I} \mathrm{Gal}(\NNum^\dagger(k_i)_F)$ and that all the homomorphisms in this filtered diagram are injective.
\end{remark}
\section{Proof of Theorem~\ref{thm:diagram}}
Let us assume first that the field extension $\overline{k}/k$ is finite. Consider the following compositions (see diagram \eqref{eq:diag-com-dagger})
\begin{eqnarray*}
\mathrm{Tate}(k)_F \subset \Num^\dagger(k)_F \stackrel{\Phi_\cN}{\too} \NNum^\dagger(k)_F &
\mathrm{Tate}(\overline{k})_F \subset \Num^\dagger(\overline{k})_F \stackrel{\Phi_\cN}{\too} \NNum^\dagger(\overline{k})_F\,,&
\end{eqnarray*}
where $\mathrm{Tate}(-)_F$ stands for the full subcategory of Tate motives. As proved in \cite[Thm.~1.7]{Galois}, these compositions give rise to the following exact sequences 
\begin{eqnarray}
\mathrm{Gal}(\NNum^\dagger(k)_F) \too \mathrm{Gal}(\Num^\dagger(k)_F) \too \bbG_m \too 1\label{eq:exact1} \\
\mathrm{Gal}(\NNum^\dagger(\overline{k})_F) \too \mathrm{Gal}(\Num^\dagger(\overline{k})_F) \too \bbG_m \too 1\nonumber\,.
\end{eqnarray}
Hence, it remains only to show that the three squares of \eqref{eq:diagram} are commutative. The commutativity of the one on the right-hand-side follows simply from the fact that the functor $\Num^\dagger(k)_F \stackrel{\Phi_\cN}{\to} \NNum^\dagger(k)_F$ restricts to a $F$-linear $\otimes$-equivalence $\AM(k)_F \stackrel{\sim}{\to} \NAM(k)_F$; see Theorems \ref{thm:main2} and \ref{thm:main3}. On the other hand, the commutativity of the lower left-hand-side square follows automatically from the commutativity of \eqref{eq:diag-com-dagger}. We now claim that the upper left-hand-side square of \eqref{eq:diagram} is also commutative. Recall from \cite[\S11]{Galois} that the motivic Galois groups of $\mathrm{Tate}(k)_F$ and $\mathrm{Tate}(\overline{k})_F$ (which are both $\bbG_m$) are completely determined by their evaluation at the Tate motive $\bbQ(1)$. Hence, our claim follows immediately from the fact that the base-change functor
\begin{eqnarray*}
\Num^\dagger(k)_F \too \Num^\dagger(\overline{k})_F && Z \mapsto Z_{\overline{k}}
\end{eqnarray*}
maps the Tate motive to itself. 

Let us now assume that the field extension $\overline{k}/k$ is infinite. The same arguments as above show that \eqref{eq:exact1} is exact and the right-hand-side square of \eqref{eq:diagram} is commutative. Let us now construct (via a limit procedure) the remaining of diagram \eqref{eq:diagram}. Consider the filtrant diagram $\{k_i\}_{i\in I}$ of intermediate field extensions $\overline{k}/k_i/k$ which are finite over $k$, and the associated filtrant diagram of $F$-linear $\otimes$-functors
\begin{equation}\label{eq:filtrant1}
\{\mathrm{Tate}(k_i)_F \subseteq \Num^\dagger(k_i)_F \stackrel{\Phi_\cN^i}{\too} \NNum^\dagger(k_i)_F \}_{i \in I}\,.
\end{equation}
As explained above, \eqref{eq:filtrant1} gives rise to a filtrant diagram of exact sequences
\begin{equation}\label{eq:filtrant2}
\{\mathrm{Gal}(\NNum^\dagger(k_i)_F) \too \mathrm{Gal}(\Num^\dagger(k_i)_F) \too \bbG_m \too 1 \}_{i \in I}\,.
\end{equation}
Now, recall from the proof of the inclusion $\mathrm{Ker}(P) \subseteq \mathrm{Im}(I)$ of Theorem~\ref{thm:main4} (with $\overline{k}/k$ infinite), that all the transition homomorphisms of the above diagram \eqref{eq:filtrant2} are injective. Moreover, $\mathrm{Gal}(\NNum^\dagger(\overline{k})_F)$ (resp. $\mathrm{Gal}(\Num^\dagger(\overline{k})_F)$) identifies with the limit of the filtrant diagram $\{\mathrm{Gal}(\NNum^\dagger(k_i)_F)\}_{i \in I}$ (resp. of $\{\mathrm{Gal}(\Num^\dagger(k_i)_F)\}_{i \in I}$). Hence, by passing to the limit in \eqref{eq:filtrant2}, one obtains the following commutative diagram
$$
\xymatrix@C=2.5em@R=2em{
\mathrm{Gal}(\NNum^\dagger(\overline{k})_F) \ar[d]_-I \ar[r] & \mathrm{Gal}(\Num^\dagger(\overline{k})_F)  \ar[d]_-I \ar[r]& \bbG_m \ar[r]  \ar@{=}[d] & 1\\
\mathrm{Gal}(\NNum^\dagger(k)_F) \ar[r] & \mathrm{Gal}(\Num^\dagger(k)_F) \ar[r]& \bbG_m \ar[r] & 1
\,,}
$$
where the upper row is an exact sequence. This concludes the proof.
\section{Proof of Theorem~\ref{thm:main5}}
We start with item (i). Recall from \cite[Prop.~4.1]{Weight} that the category of noncommutative Chow motives fully-embeds into the triangulated category of noncommutative mixed motives. Concretely, we have a fully-faithful $\bbQ$-linear $\otimes$-functor $\Upsilon$ making the following diagram commute
\begin{equation}\label{eq:com-triangle}
\xymatrix{
\sdgcat(k) \ar[d]_-{\cU} \ar[dr]^-{\cU_M} & \\
\NChow(k)_\bbQ \ar[r]_-{\Upsilon} & \Mix(k)_\bbQ\,.
}
\end{equation}
This allows us to consider the composed $\bbQ$-linear $\otimes$-functor
$$ \Omega: \AM(k)_\bbQ \subset \Chow(k)_\bbQ \stackrel{\Phi}{\too} \NChow(k)_\bbQ \stackrel{\Upsilon}{\too} \Mix(k)_\bbQ\,.$$
By Theorem~\ref{thm:main2} we observe that $\Omega$ is fully-faithful. Moreover, since $\NMAM(k)_\bbQ$ is by definition the smallest thick triangulated subcategory of $\Mix(k)_\bbQ$ generated by the objects $\cU_M(\cD_\perf^\dg(Z))$, with $Z$ a finite {\'e}tale $k$-scheme, one concludes from the commutative diagram \eqref{eq:com-triangle} that $\Omega$ factors through the inclusion $\NMAM(k)_\bbQ \subset \Mix(k)_\bbQ$. Now, as proved in \cite[Prop.~1.4]{Wildehaus} (see also \cite[page 217]{Voevodsky}) the category $\MAM(k)_\bbQ$ is $\otimes$-equivalence to the bounded derived category $\cD^b(\AM(k)_\bbQ)$. Moreover, since $\AM(k)_\bbQ$ is abelian semi-simple, $\cD^b(\AM(k)_\bbQ)$ identifies with the category $\mathrm{Gr}^b_\bbZ(\AM(k)_\bbQ)$ of bounded $\bbZ$-graded objects in $\AM(k)_\bbQ$. Hence, the above functor $\Omega$ admits the following natural extension:
\begin{eqnarray*}
\Psi: \MAM(k)_\bbQ \too \NMAM(k)_\bbQ \subset \Mix(k)_\bbQ && \{N_n\}_{n\in \bbZ} \mapsto \bigoplus_{n \in \bbZ} \Omega(N_n)[n]\,.
\end{eqnarray*}
By construction this extension is $\bbQ$-linear, $\otimes$-triangulated, and moreover faithful. Now, consider the following commutative diagram
$$
\xymatrix{
*+<2.5ex>{\{\textrm{finite {\'e}tale $k$-schemes}\}^\op} \ar[d]^-{M} \ar@{^{(}->}[r] \ar@/_2pc/[dd]_-{M_{gm}^\vee}& \SmProj(k)^\op \ar[d]_-{M} \ar[r]^-{\cD_\perf^\dg(-)} & \sdgcat(k) \ar[d]^-{\cU} \ar[dr]^-{\cU_M} & \\
*+<2.5ex>{\AM(k)_\bbQ} \ar[d]^-R \ar@{^{(}->}[r] & \Chow(k)_\bbQ \ar[r]_-{\Phi} & \NChow(k)_\bbQ \ar[r]^-{\Upsilon} & \Mix(k)_\bbQ \\
\MAM(k)_\bbQ \ar[rr]_-{\Psi} && *+<2.5ex>{\NMAM(k)_\bbQ} \ar@{^{(}->}[ur] & \,,
}
$$
where $R$ is the classical fully-faithful $\bbQ$-linear $\otimes$-functor relating Chow motives with mixed motives; see \cite[\S18.3-18.4]{Andre}\cite[Corollary~2]{Voevodsky2}. This functor identifies $\AM(k)_\bbQ$ with the zero-graded objects of $\mathrm{Gr}^b_\bbZ(\AM(k)_\bbQ)\simeq \MAM(k)_\bbQ$ and when pre-composed with $M$ it agrees with $M_{gm}^\vee$. The proof of item (i) now follows from the fact that \eqref{eq:diag-mixed} is the outer square of the above commutative diagram.

Let us now prove item (ii). Since $\Psi$ is triangulated, $\MAM(k)_\bbQ$ is generated by the objects $M_{gm}^\vee(Z)$, with $Z$ a finite {\'e}tale $k$-scheme, and $\NMAM(k)_\bbQ$ is generated by their images under $\Psi$, it suffices to show that the induced homomorphisms
\begin{equation}\label{eq:induced-hom}
\Hom(M_{gm}^\vee(Z_1),M_{gm}^\vee(Z_2)[-n])\too \Hom(\Psi M_{gm}^\vee(Z_1),\Psi M_{gm}^\vee(Z_2)[-n])
\end{equation}
are isomorphisms for all finite {\'e}tale $k$-schemes $Z_1$ and $Z_2$ and integers $n \in \bbN$. By construction of $\Psi$ (and Theorem~\ref{thm:main2}), \eqref{eq:induced-hom} is an isomorphism for $n=0$. As explained above, $\MAM(k)_\bbQ$ identifies with $\mathrm{Gr}^b_\bbZ(\AM(k)_\bbQ)$ and so the left hand-side of \eqref{eq:induced-hom} is trivial for $n \neq 0$. Hence, it remains to show that the right hand-side of \eqref{eq:induced-hom} is also trivial for $n \neq 0$. By combining the commutativity of the square \eqref{eq:diag-mixed} with \cite[Prop.~9.3]{CT1}, we obtain the following computation
\begin{equation}\label{eq:computation}
\Hom_{\NMAM(k)_\bbQ}(\Psi M_{gm}^\vee(Z_1), \Psi M_{gm}^\vee (Z_2)[-n]) \simeq  \left\{ \begin{array}{ll} K_n(Z_1 \times Z_2)_\bbQ &  n\geq 0 \\ 0 & n <0 \,. \end{array} \right.
\end{equation}
Now, recall that $Z_1=\mathrm{spec}(A_1)$ and $Z_2=\mathrm{spec}(A_2)$ for some finite dimensional {\'e}tale $k$-algebras $A_1$ and $A_2$. As a consequence, $K_n(Z_1\times Z_2)_\bbQ = K_n(A_1\otimes_kA_2)_\bbQ$. When $k=\bbF_q$ we conclude then that $A_1$ and $A_2$ (and hence $A_1\otimes_{\bbF_q}A_2$) have a finite number of elements. Since the algebraic $K$-theory of the $\bbF_q$-algebra $A_1\otimes_{\bbF_q}A_2$ is the same as the algebraic $K$-theory of the underlying ring, we conclude from \cite[Prop.~1.16]{Weibel-Book} that all the groups $K_n(A_1 \otimes_{\bbF_q}A_2), n \geq 1$, are finite and hence torsion groups. As a consequence, $K_n(Z_1\times Z_2)_\bbQ=0$ for $n \geq 1$ and so $\Psi$ is an equivalence. The equivalence $\NMAM(\bbF_q)_\bbQ \simeq \cD^b(\mathrm{Rep}_\bbQ(\widehat{\bbZ}))$ follows now from the $\bbQ$-linearization of the Galois-Grothendieck correspondence $\AM(\bbF_q)_\bbQ \simeq \mathrm{Rep}_\bbQ(\mathrm{Gal}(\overline{\bbF_q}/\bbF_q))$ (described in the proof of Theorem~\ref{thm:main3}) and from the standard isomorphism $\mathrm{Gal}(\overline{\bbF_q}/\bbF_q)\simeq \widehat{\bbZ}$. This concludes the proof of item (ii). In what concerns item (iii), by taking $Z_1=Z_2 =\mathrm{spec}(\bbQ)$ and $n=1$ in the above computation \eqref{eq:computation}, one obtains
\begin{equation}\label{eq:group}
K_1(\mathrm{spec}(k)\times \mathrm{spec}(k))_\bbQ \simeq K_1(\mathrm{spec}(k))_\bbQ \simeq K_1(k)_\bbQ \simeq k^\times\otimes_\bbZ \bbQ\,,
\end{equation}
where $k^\times$ denotes the multiplicative group of $k$. Since by hypothesis $\bbQ\subseteq k$, we conclude then that \eqref{eq:group} is not trivial. As a consequence, the functor $\Psi$ is {\em not} full since the left hand-side of \eqref{eq:induced-hom} is trivial for $n \neq 0$. The isomorphisms of item(iii) follow from the combination of \eqref{eq:computation} (with $k=\bbQ, Z_1=Z_2=\mathrm{spec}(\bbQ)$) with the isomorphisms $K_n(\bbQ)\simeq \bbZ$ for $n \equiv 5$ (mod $8$); see \cite[page~140]{Weibel-Handbook}. This achieves the proof.

\appendix
\section{Short exact sequences of Galois groups}\label{appendix:Galois}
In this appendix we establish a general result about short exact sequences of Galois groups. As an application we obtain a new proof of Deligne-Milne's short exact sequence \eqref{eq:shortexact}; see \S\ref{sub:new-proof}.

Let $K$ be a field, $\cA\subset \cC$ and $\cD$ three $K$-linear abelian idempotent complete rigid symmetric monoidal categories, $H:\cC \to \cD$ a $K$-linear exact $\otimes$-functor, and $\omega: \cD \to \mathrm{Vect}(K)$ a fiber functor. Out of this data one constructs the Galois groups
\begin{eqnarray*}
\mathrm{Gal}(\cD):= \underline{\mathrm{Aut}}^\otimes(\omega) & \mathrm{Gal}(\cC):= \underline{\mathrm{Aut}}^\otimes(\omega\circ H) & \mathrm{Gal}(\cA):= \underline{\mathrm{Aut}}^\otimes((\omega\circ H)_{|\cA})
\end{eqnarray*}
as well as the group homomorphisms
\begin{eqnarray*}
I:\mathrm{Gal}(\cD) \stackrel{\varphi \mapsto \varphi \circ H}{\too} \mathrm{Gal}(\cC)  &&
P:\mathrm{Gal}(\cC) \stackrel{\psi \mapsto \psi_{|\cA}}{\too} \mathrm{Gal}(\cA)\,;
\end{eqnarray*}
see \cite[\S2.3]{Andre}. Now, consider the following assumptions:
\begin{itemize}
\item[(A1)] Assume that for every object $X$ of $\cA$ there exists an integer $m$ (which depends on $X$) such that $H(X) \simeq \oplus_{i=1}^m {\bf 1}_\cD$, where ${\bf 1}_\cD$ denotes the $\otimes$-unit of $\cD$.
\item[(A2)] Assume that $\cD$ is semi-simple and that for every object $Y$ of a (fixed) dense subcategory $\cD_{\mathrm{dens}}$ of $\cD$ there exists an object $X \in \cC$ (which depends on $Y$) and a surjective morphism $H(X) \twoheadrightarrow Y$.
\item[(A3)] Assume that $\cD$ is semi-simple and that $H$ fits into a monadic adjunction
\begin{equation}\label{eq:monadic}
\xymatrix{
\cD \ar@<1ex>[d]^-G \\
\cC\ar@<1ex>[u]^-H\,.
}
\end{equation}
Assume moreover that $G({\bf 1}_\cD) \in \cA$ and that the natural transformation
$$ \id \otimes \epsilon_{{\bf 1}_\cD}: H(-)\otimes HG({\bf 1}_\cD) \Rightarrow H(-)\otimes {\bf 1}_{\cD} \simeq H(X)$$
gives rise by adjunction to an isomorphism $-\otimes G({\bf 1}_\cD) \stackrel{\sim}{\Rightarrow} (G \circ H)(-)$ of functors, where $\epsilon$ denotes the counit of \eqref{eq:monadic}.
\item[(A3')] Assume that there exists a filtrant diagram $\{H: \cC \stackrel{H_i}{\to} \cD_i \to \cD \}_{i \in I}$ such that $\mathrm{colim}_i\cD_i\simeq \cD$ and on which all the functors $H_i$ satisfy assumption (A3).
\end{itemize}
\begin{theorem}\label{thm:general}
Under the above notations the following holds:
\begin{itemize}
\item[(i)] The group homomorphism $P$ is surjective.
\item[(ii)] Assumption (A1) implies that the composition $P \circ I$ is trivial.
\item[(iii)] Assumption (A2) implies that $I$ is injective.
\item[(iv)] Assumption (A3), or more generally (A3'), implies that $\mathrm{Ker}(P) \subseteq \mathrm{Im}(I)$.
\end{itemize}
In particular, if all the above assumptions hold one obtains a well-defined short exact sequence of Galois groups
$$ 1 \too \mathrm{Gal}(\cD) \stackrel{I}{\too} \mathrm{Gal}(\cC) \stackrel{P}{\too} \mathrm{Gal}(\cA) \too 1\,.$$
\end{theorem}
Due to its simplicity and generality we believe that Theorem~\ref{thm:general} can be useful in other applications to Galois groups and Tannakian categories. When applied to $K:=F$, $\cA:= \NAM(k)_F$, $\cC=\NNum^\dagger(k)_F$, $\cD:=\NNum^\dagger(\overline{k})_F$, $\omega:= \overline{HP_\ast}$, and $H:=-\otimes_k \overline{k}$, Theorem  ~\ref{thm:general} (as well as its proof) reduces to Theorem~\ref{thm:main4}. 
\begin{proof}
Since by hypothesis $\cA$ is a full subcategory of $\cC$, item (i) follows from the Tannakina formalism; see \cite[\S2.3.3]{Andre}. Let $\varphi$ be a $\otimes$-automorphism of the fiber functor $\omega: \cD \to \mathrm{Vect}(K)$, \ie an element of $\mathrm{Gal}(\cD)$. Under these notations, item (ii) follows automatically from the equalities 
$$ \varphi_{\oplus^m_{i=1}{\bf 1}_\cD} = \oplus^m_{i=1} \varphi_{{\bf 1}_\cD}=\oplus_{i=1}^m \id_{\omega({\bf 1}_\cD)}= \id_{\omega(\oplus_{i=1}^m {\bf 1}_\cD)}\,.$$
Item (iii) is the content of the general Proposition~\ref{prop:monomorphism}. In what concerns item (iv), its proof is similar to the proof of the inclusion $\mathrm{Ker}(P) \subseteq \mathrm{Im}(I)$ of Theorem~\ref{thm:main4}. Simply run the same argument and use the assumption (A3) (resp. (A3')) instead of the assumption that the field extension $\overline{k}/k$ is finite (resp. infinite).
\end{proof}
\subsection{A new proof of \eqref{eq:shortexact}}\label{sub:new-proof}
Recall from Deligne-Milne \cite[page~63]{DelMil}, the proof of the following short exact sequence
\begin{equation}\label{eq:ses-last}
1 \too \mathrm{Gal}(\Num^\dagger(\overline{k})_F) \stackrel{I}{\too} \mathrm{Gal}(\Num^\dagger(k)_F)\stackrel{P}{\too} \mathrm{Gal}(\overline{k}/k) \too 1\,.
\end{equation}
As in Theorem~\ref{thm:general}(i) the subjectivity of $P$ follows from the Tannakian formalism. The proof of the injectivity of $I$ can also be considered as an application of Theorem~\ref{thm:general}(ii): as explained in {\em loc. cit.}, every smooth projective $\overline{k}$-scheme $Z$ admits a model $Z_0$ over a finite field extension $k_0/k$ and consequently one obtains a surjective morphism $(\mathrm{Res}_{k_0/k}(Z_0))_{\overline{k}} \twoheadrightarrow Z$. Our general Proposition~\ref{prop:monomorphism} should then be considered as an axiomatization of this geometric argument.

On the other hand, the proof of the exactness of \eqref{eq:ses-last} at the middle term is based on the interpretation of $\Hom(N_{\overline{k}},N'_{\overline{k}})$ as an Artin motive for any two objects $N, N' \in \Num^\dagger(k)_F$. This is the key step in Deligne-Milne's proof that doesn't seem to admit a noncommutative analogue. In order to prove this exactness one can alternative apply Theorem~\ref{thm:general} to $K:=F$, $\cA:=\AM(k)_F$, $\cC:= \Num^\dagger(k)_F$, $\cD:= \Num^\dagger(\overline{k})_F$, $\omega:= H_{dR}(-)$, and to the base-change functor $H:= \Num^\dagger(k)_F \to \Num^\dagger(\overline{k})_F, Z \mapsto Z_{\overline{k}}$. The verification of (A3) and (A3') is similar to the one made in the proof of Theorem~\ref{thm:main4}.

\end{document}